\documentclass[11pt]{amsart}
\usepackage{amssymb,amsfonts,amscd,tikz,enumerate,mathtools,tikz-cd}
\usepackage[utf8]{inputenc}
\usepackage[colorlinks,linktocpage]{hyperref}
\hypersetup{linkcolor=[rgb]{0,0,0.715}}
\hypersetup{citecolor=[rgb]{0,0.715,0}}

\newtheorem{theorem}{Theorem}[section]
\newtheorem{corollary}[theorem]{Corollary}

\newtheorem{lemma}[theorem]{Lemma}
\newtheorem{conjecture}[theorem]{Conjecture}

\newtheorem*{thm1}{Main Theorem}

\theoremstyle{definition}
\newtheorem{definition}[theorem]{Definition}

\newtheorem{example}[theorem]{Example}

\theoremstyle{remark}
\newtheorem{remark}[theorem]{Remark}

\makeatletter
\let\c@equation\c@theorem
\makeatother
\numberwithin{equation}{section}

\title[]{On the limit of simply connected manifolds with discrete isometric cocompact group actions}
\author{Jikang Wang}
\address{Department of Mathematics, University of California, Berkeley, CA 94720, USA}

\address{Current address: Institute of Geometry and Physics, University of Science and Technology of China, Hefei, Anhui, China, 230026}
\email{wangjikang@ustc.edu.cn}
\thanks{}
\thanks{}

\thanks{}

\DeclareUnicodeCharacter{03C3}{\ensuremath{\sigma}}
\DeclareUnicodeCharacter{2217}{\ensuremath{*}}
\begin{document}
\date{}

\begin{abstract}
We study the fundamental group of the equivariant Gromov--Hausdorff limit space
\[
(M_i, p_i, G_i) \xrightarrow{eGH} (Y, p, G),
\]
where $\{(M_i,p_i)\}_{i\in\mathbb N}$ is a sequence of complete, connected, and simply connected $n$-dimensional Riemannian manifolds with $\mathrm{Ric}\ge -(n-1)$. Here $G_i\subset \mathrm{Iso}(M_i)$ is a discrete subgroup satisfying $\mathrm{diam}(M_i/G_i)\le D$ for some fixed constant $D>0$.

Let $G_0$ denote the identity component of $G$. Then $G_0$ is a nilpotent Lie group. Let $T^k \subseteq G_0$ be the maximal torus.

Our main result asserts that the quotient space $Y/T^k$ is simply connected. Moreover, the fundamental group $\pi_1(Y,p)$ is generated, up to conjugation, by loops contained in $T^k$-orbits. More precisely, each generator can be represented by a loop of the form
\begin{equation*}
\gamma = \alpha \cdot \beta \cdot \alpha^{-1},
\end{equation*}
where $\alpha$ is a curve from $p$ to some point $y \in Y$, and $\beta$ is a loop based at $y$ that lies entirely in the $T^k$-orbit of $y$.
\end{abstract}
\maketitle
\tableofcontents
\section{Introduction}

Consider a sequence of complete, connected $n$-dimensional Riemannian manifolds $(M_i,p_i)$ satisfying the Ricci curvature lower bound $\mathrm{Ric} \ge -(n-1)$. After passing to a subsequence if necessary, we may assume that $(M_i,p_i)$ converges in the pointed Gromov--Hausdorff sense to a proper geodesic metric space $(Y,p)$, referred to as a \textit{Ricci limit space}.

Let $G_i$ be a closed subgroup of the isometry group $\mathrm{Iso}(M_i)$. After passing to a further subsequence, we obtain equivariant Gromov--Hausdorff convergence
\[
(M_i,p_i,G_i) \xrightarrow{eGH} (Y,p,G).
\]

The regularity and geometric structure theory of Ricci limit spaces has been extensively developed in the foundational works of Cheeger, Colding, and Naber \cite{CheegerColding1997, CheegerColding2000a, CheegerColding2000b, CheegerNaber2013, CheegerNaber2015, ColdingNaber2012}. In particular, a celebrated theorem of Cheeger--Colding and Colding--Naber asserts that the isometry group of a Ricci limit space is a Lie group.

More recently, the local topology of Ricci limit spaces has been studied in \cite{SormaniWei2001, SormaniWei2004, PanWei2019, PanWang2021, Wang2021}. In particular, $Y$ is known to be semi-locally simply connected. The following result describes the structure of the fundamental group of the limit space in the bounded diameter setting.

\begin{theorem}[\cite{SormaniWei2001, Wang2021}]\label{thm:bounded_diam}
Let $\{(M_i,p_i)\}_{i\in\mathbb{N}}$ be a sequence of compact, connected $n$-dimensional Riemannian manifolds satisfying $\mathrm{diam}(M_i)\le D$ and $\mathrm{Ric}\ge -(n-1)$. If
\[
(M_i,p_i) \xrightarrow{GH} (Y,p),
\]
then for all sufficiently large $i$, there exists a surjective homomorphism
\[
s_i:\pi_1(M_i,p_i)\to \pi_1(Y,p).
\]
In particular, if each $M_i$ is simply connected, then the limit space $Y$ is also simply connected.
\end{theorem}

In the absence of a diameter bound, the limit space $Y$ may be non-compact, and a global Gromov--Hausdorff approximation between $M_i$ and $Y$ need not exist. Consequently, Theorem~\ref{thm:bounded_diam} may fail without the bounded diameter assumption; see \cite{SormaniWei2004, Zamora2020} or Example~\ref{ex:su2_collapse} for counterexamples.

These observations naturally motivate the following question: under what conditions does a sequence of simply connected, possibly non-compact manifolds with a Ricci curvature lower bound converge to a simply connected limit space?

Let $G_i$ be a discrete subgroup of the isometry group $\mathrm{Iso}(M_i)$. We say that the sequence $\{(M_i,p_i,G_i)\}_{i\in\mathbb{N}}$ is \textit{almost homogeneous} if the diameters of the quotient spaces tend to zero, that is,
$\mathrm{diam}(M_i/G_i)\to 0 $ as $i\to\infty$.
Using the fact that Ricci limit spaces are semi-locally simply connected \cite{Wang2022}, Zamora proved that almost homogeneity implies that the limit space $Y$ is simply connected.

\begin{theorem}[\cite{Zamora2020}]\label{thm:almost_homo}
Let $(M_i,p_i)$ be a sequence of complete, connected, simply connected Riemannian $n$-manifolds with $\mathrm{Ric}\ge -(n-1)$, and let $G_i\subset \mathrm{Iso}(M_i)$ be a discrete subgroup such that $\mathrm{diam}(M_i/G_i)\to 0$ as $i\to\infty$. Suppose that
\[
(M_i,p_i,G_i) \xrightarrow{eGH} (Y,p,G).
\]
Then the Ricci limit space $Y$ is simply connected.
\end{theorem}



Zamora conjectured that the vanishing diameter condition $\mathrm{diam}(M_i/G_i)\to 0$ can be relaxed to a uniform bounded diameter assumption $\mathrm{diam}(M_i/G_i)\le D$.

\begin{conjecture}[\cite{Zamora2020}]\label{conj:zamora}
Let $(M_i,p_i)$ be a sequence of complete, connected, simply connected Riemannian $n$-manifolds satisfying $\mathrm{Ric}\ge -(n-1)$. Let $G_i\subset \mathrm{Iso}(M_i)$ be a discrete subgroup such that $\mathrm{diam}(M_i/G_i)\le D$ for some uniform constant $D>0$. Suppose that
\[
(M_i,p_i,G_i)\xrightarrow{eGH}(Y,p,G).
\]
Then the limit space $Y$ is simply connected.
\end{conjecture}

The following example shows that the discreteness of the subgroups $G_i$ in Theorem~\ref{thm:almost_homo} and Conjecture~\ref{conj:zamora} is necessary.

\begin{example}\label{ex:su2_collapse}
Consider the Lie group $SU(2)\cong S^3$, whose Lie algebra $\mathfrak{su}(2)$ is generated by the standard basis
\[
X_1=\begin{pmatrix} i & 0 \\ 0 & -i \end{pmatrix}, \qquad
X_2=\begin{pmatrix} 0 & 1 \\ -1 & 0 \end{pmatrix}, \qquad
X_3=\begin{pmatrix} 0 & i \\ i & 0 \end{pmatrix}.
\]
These induce left-invariant vector fields on $S^3$. Let $g$ be the bi-invariant metric for which $\{X_1,X_2,X_3\}$ forms an orthonormal basis; this is the standard round metric on the unit sphere.

We construct a sequence of left-invariant metrics $g_k$ by rescaling the directions orthogonal to the Hopf fibers. Specifically, we declare $\{X_1,\tfrac{1}{k}X_2,\tfrac{1}{k}X_3\}$ to be an orthonormal basis for $g_k$, where $k\to\infty$. Equivalently,
\[
g_k(X_1,X_1)=1, \qquad g_k(X_2,X_2)=g_k(X_3,X_3)=k^2.
\]
As $k\to\infty$, the sectional curvatures of $(S^3,g_k)$ remain bounded below.

Geometrically, this sequence corresponds to the Hopf fibration $S^1\hookrightarrow S^3\to S^2$. The length of the $S^1$ fibers (generated by $X_1$) remains constant, while the base $S^2$ expands. Choosing a base point $p\in S^3$, we obtain pointed Gromov--Hausdorff convergence
\[
((S^3,g_k),p)\xrightarrow{pGH}(\mathbb{R}^2\times S^1,q),
\]
where the limit space is flat. In particular, $\mathbb{R}^2\times S^1$ is not simply connected, even though each $(S^3,g_k)$ is simply connected.

Finally, $SU(2)$ acts transitively and isometrically on each $(S^3,g_k)$. This demonstrates that the discreteness assumption in Theorem~\ref{thm:almost_homo} and Conjecture~\ref{conj:zamora} is necessary.
\end{example}


In \cite{PanRong2018}, Pan and Rong confirmed Conjecture~\ref{conj:zamora} under the stronger curvature condition $\mathrm{Ric} \ge 0$. Their proof relies on the \textit{Cheeger--Gromoll splitting theorem}, which is generally unavailable under a negative Ricci curvature lower bound.

Our main theorem provides a partial resolution of Zamora's conjecture. We formulate the result in the more general setting of proper geodesic metric spaces.

\begin{thm1}\label{thm:main_A}
Let $\{(X_i,p_i)\}_{i\in\mathbb{N}}$ be a sequence of simply connected proper geodesic metric spaces. Suppose $G_i\subset \mathrm{Iso}(X_i)$ is a discrete subgroup satisfying $\mathrm{diam}(X_i/G_i)\le D$ for some fixed constant $D>0$. Assume that the sequence converges in the equivariant Gromov--Hausdorff sense:
\[
(X_i,p_i,G_i)\xrightarrow{eGH}(X,p,G).
\]
Furthermore, assume that $X$ is proper and semi-locally simply connected, and that $G$ is a Lie group.

Then the following hold:
\begin{enumerate}
\item The identity component $G_0$ of $G$ is a nilpotent Lie group.

\item Let $T^k\subseteq G_0$ denote the maximal torus. Then the quotient space $X/T^k$ is simply connected.

\item The fundamental group $\pi_1(X,p)$ is generated, up to conjugation, by loops contained in $T^k$-orbits. More precisely, each generator can be represented by a loop of the form
\[
\alpha\cdot\beta\cdot\alpha^{-1},
\]
where $\alpha$ is a path from $p$ to some point $y\in X$, and $\beta$ is a loop based at $y$ that lies entirely in the orbit $T^k\cdot y$.
\end{enumerate}
\end{thm1}

Since any Ricci limit space $Y$ is proper and semi-locally simply connected, and $\mathrm{Iso}(Y)$ is a Lie group, the Main Theorem yields a partial resolution of Conjecture~\ref{conj:zamora} by showing that $Y/T^k$ is simply connected.

It is also worth noting that any $\mathrm{RCD}(K,N)$ space is proper and semi-locally simply connected, and its isometry group is a Lie group; see \cite{GuSan2019, MondinoWei2019, Sosa2018, Wang2022}. Consequently, the Main Theorem applies equally in the $\mathrm{RCD}$ setting.

Finally, we observe that if the discreteness assumption on $G_i$ is removed from the Main Theorem, one can still conclude that the quotient $X/G_0$ is simply connected.

\begin{theorem}\label{thm:G_0}
Let $\{(X_i,p_i)\}_{i\in\mathbb{N}}$ be a sequence of simply connected proper geodesic metric spaces. Suppose $G_i\subseteq \mathrm{Iso}(X_i)$ is a closed subgroup satisfying $\mathrm{diam}(X_i/G_i)\le D$ for some fixed $D>0$. Assume that the sequence converges in the equivariant Gromov--Hausdorff sense:
\[
(X_i,p_i,G_i)\xrightarrow{eGH}(X,p,G).
\]
If $X$ is proper and semi-locally simply connected, and $G$ is a Lie group, then the quotient space $X/G_0$ is simply connected.
\end{theorem}

The paper is organized as follows. Section~2 reviews the necessary preliminaries. In Section~3, we extend a key result of Fukaya and Yamaguchi \cite[Theorem~3.10]{FukayaYamaguchi1994} by removing the assumption that the subgroups $G_i$ are free or discrete. In particular, we establish the following theorem. (See Section~\ref{2.1} for the definition of the notation $G(p,r)$.)

\begin{theorem}\label{thm:normal_subgroups}
Let $\{(X_i,p_i,G_i)\}_{i\in\mathbb{N}}$ be a sequence of simply connected proper geodesic metric spaces, where each $G_i\subseteq \mathrm{Iso}(X_i)$ is a closed subgroup. Suppose that
\[
(X_i,p_i,G_i)\xrightarrow{eGH}(X,p,G),
\]
and assume that there exists a closed normal subgroup $G'\trianglelefteq G$ satisfying:
\begin{enumerate}
\item $G/G'$ is discrete;
\item $G'$ is generated by $G'(p,R)$ for some $R>0$;
\item $\mathrm{diam}(X_i/G_i)\le D$ for some fixed constant $D>0$.
\end{enumerate}
Then there exists a sequence of normal subgroups $G_i'\trianglelefteq G_i$ such that:
\begin{enumerate}
\item $G_i'$ converges to $G'$ in the equivariant sense;
\item $G_i/G_i' \cong G/G'$ for all sufficiently large $i$;
\item $G_i'$ is generated by $G_i'(p_i,R+\epsilon_i)$, where $\epsilon_i\to 0$;
\item $G/G'$ is finitely presented.
\end{enumerate}
\end{theorem}

In Section~4, we establish an extension property that enables the construction of a global map which acts as an equivariant Gromov--Hausdorff approximation (eGHA) on each $D$-ball. This technical result is crucial for passing from local to global topological control.

We say that a map $f':X \to Y$ is \emph{$\epsilon$-injective} if for any $x, y \in X$, the condition $d(f'(x),f'(y)) \le \epsilon$ implies $d(x,y) \le 3\epsilon$. Any $\epsilon$-GHA is $\epsilon$-injective on a ball centered at the base point.

\begin{theorem}\label{thm:extend}
Let $(X,p,G)$ and $(Y,q,G)$ be proper geodesic metric spaces equipped with isometric $G$-actions. Suppose that for some constants $D>10\epsilon>0$ with $D < \tfrac{1}{10\epsilon}$, there exists an $\epsilon$-Gromov--Hausdorff approximation
\[
f: B_{4D}(p)\to B_{4D}(q)
\]
that is $\epsilon$-almost $G$-equivariant. That is, for any $x\in B_{4D}(p)$ and $g\in G$ such that $gx\in B_{4D}(p)$, we have
\[
d\big(gf(x),f(gx)\big)\le \epsilon.
\]
Assume further that the quotient diameters satisfy $\mathrm{diam}(X/G)\le D$ and $\mathrm{diam}(Y/G)\le D$. Then there exists a global map $f':X\to Y$ such that:
\begin{enumerate}
\item $f'$ is $2\epsilon$-close to $f$ on $B_{4D}(p)$;
\item $f'$ is a $2\epsilon$-Gromov--Hausdorff approximation on every $D$-ball in $X$;
\item $f'$ is $2\epsilon$-almost $G$-equivariant on $X$;
\item $f'$ is $\epsilon$-injective, and the image $f'(X)$ is $3\epsilon$-dense in $Y$.
\end{enumerate}
\end{theorem}

Next, we establish a lemma concerning the topology of quotient spaces; see \cite{Armstrong} for related results.

\begin{lemma}\label{lem:quotient_topology}
Let $X$ be a simply connected proper geodesic metric space, and suppose that $\mathrm{Iso}(X)$ is a Lie group. Let $G \subseteq \mathrm{Iso}(X)$ be a closed subgroup generated by the identity component $G_0$ together with all isotropy subgroups. Then the quotient space $X/G$ is simply connected.
\end{lemma}

By combining Theorem~\ref{thm:normal_subgroups}, Theorem~\ref{thm:extend}, Lemma~\ref{lem:quotient_topology}, and Lemma~\ref{lem:limit_simply_connected}, we complete the proof of Theorem~\ref{thm:G_0} in Section~4.

While the results in Sections~3 and~4 hold for arbitrary closed subgroups $G_i\subseteq \mathrm{Iso}(X_i)$, the discreteness of $G_i$ becomes essential for the structural results in the final part of the paper. In particular, we establish a Margulis-type lemma, showing that if each $G_i$ is discrete and the limit group $G$ is a Lie group, then the identity component $G_0$ is necessarily nilpotent.

\begin{lemma}[Margulis-type Lemma]\label{lem:nilpotent_G0}
Let $\{(X_i,p_i,G_i)\}_{i\in\mathbb{N}}$ be a sequence of proper geodesic metric spaces, where each $G_i$ is a discrete subgroup of $\mathrm{Iso}(X_i)$. Suppose that
\[
(X_i,p_i,G_i)\xrightarrow{eGH}(X,p,G),
\]
where the limit group $G$ is a Lie group. Then the identity component $G_0$ is a nilpotent Lie group.
\end{lemma}

In Section~6, we establish the following lifting lemma, which relates the fundamental group of a space to that of its quotient under a connected Lie group action.

\begin{lemma}\label{lem:lifting_compact}
Let $(Y,p)$ be a semi-locally simply connected proper geodesic metric space, and suppose that a connected Lie group $K$ acts isometrically on $Y$, with quotient $(Y/K, \bar{p})$. Then there is a natural surjective homomorphism
\[
\pi_* \colon \pi_1(Y,p) \longrightarrow \pi_1(Y/K, \bar{p}),
\]
whose kernel is generated, up to conjugation, by loops contained entirely in the $K$-orbits.
\end{lemma}

Under the hypotheses of the Main Theorem, we have shown that $X/G_0$ is simply connected and that $G_0$ is a nilpotent Lie group. Let $T^k \trianglelefteq G_0$ be the maximal torus. Then $G_0/T^k$ is diffeomorphic to Euclidean space. The induced $G_0/T^k$-action on $X/T^k$ is free; in particular, it follows that $X/T^k$ is simply connected.
Finally, we complete the proof of the Main Theorem by applying Lemma~\ref{lem:lifting_compact} to $T^k$-actions on $X$.

\textbf{Acknowledgments.} 
The author would like to thank Xiaochun Rong, Vitali Kapovitch, and Jiayin Pan for helpful discussions. The author also thanks the anonymous referees for their valuable suggestions.

\section{Preliminaries}
\subsection{Equivariant Gromov--Hausdorff convergence}\label{2.1}

We begin by reviewing the notation and basic results concerning equivariant Gromov--Hausdorff (eGH) convergence, following the framework of Fukaya and Yamaguchi \cite{Fukaya1986, FukayaYamaguchi1992}.

Let $(X,p)$ and $(Y,q)$ be proper geodesic metric spaces, and let $G$ and $H$ be closed subgroups of $\mathrm{Iso}(X)$ and $\mathrm{Iso}(Y)$, respectively. For any $r>0$, define
\[
G(p,r) \coloneqq \{ g \in G \mid d(gp,p) \le r \}, \qquad
H(q,r) \coloneqq \{ h \in H \mid d(hq,q) \le r \}.
\]
Under the compact-open topology, both $G(p,r)$ and $H(q,r)$ are compact.

For $\epsilon>0$, a (pointed) $\epsilon$-equivariant Gromov--Hausdorff approximation ($\epsilon$-eGHA) is a triple $(f, \phi, \psi)$ consisting of maps
\[
f: B_{1/\epsilon}(p) \to B_{1/\epsilon}(q), \quad
\phi: G(p,1/\epsilon) \to H(q,1/\epsilon), \quad
\psi: H(q,1/\epsilon) \to G(p,1/\epsilon),
\]
satisfying the following conditions:
\begin{enumerate}
\item $f(p)=q$, the image $f(B_{1/\epsilon}(p))$ is $\epsilon$-dense in $B_{1/\epsilon}(q)$, and for all $x_1,x_2 \in B_{1/\epsilon}(p)$,
\[
\big| d(f(x_1),f(x_2)) - d(x_1,x_2) \big| \le \epsilon.
\]

\item For all $g \in G(p,1/\epsilon)$ and $x \in B_{1/\epsilon}(p)$ such that $gx \in B_{1/\epsilon}(p)$,
\[
d\big(\phi(g)f(x),\, f(gx)\big) \le \epsilon.
\]

\item For all $h \in H(q,1/\epsilon)$ and $x \in B_{1/\epsilon}(p)$ such that $\psi(k)x \in B_{1/\epsilon}(p)$,
\[
d\big(h f(x),\, f(\psi(h)x)\big) \le \epsilon.
\]
\end{enumerate}

The equivariant Gromov--Hausdorff distance
$d_{eGH}((X,p,G),(Y,q,H))$ 
is defined as the infimum over all $\epsilon>0$ for which an $\epsilon$-eGHA exists. A sequence $(X_i,p_i,G_i)$ equivariantly converges to $(X,p,G)$, denoted 
$(X_i,p_i,G_i) \xrightarrow{eGH} (X,p,G)$, if
\[
\lim_{i \to \infty} d_{eGH}((X_i,p_i,G_i),(X,p,G)) = 0.
\]

Given an $\epsilon$-eGHA $(f,\phi,\psi)$, one can construct an admissible metric on the disjoint union $B_{1/\epsilon}(p) \sqcup B_{1/\epsilon}(q)$ such that the natural inclusions
\[
\iota_X: B_{1/\epsilon}(p) \hookrightarrow B_{1/\epsilon}(p)\sqcup B_{1/\epsilon}(q), \quad
\iota_Y: B_{1/\epsilon}(q) \hookrightarrow B_{1/\epsilon}(p)\sqcup B_{1/\epsilon}(q)
\]
are isometric embeddings. Under this metric, for any $x \in B_{1/\epsilon}(p)$, we have $d(x,f(x)) \le 2\epsilon$. Throughout, whenever the eGH distance is sufficiently small, we implicitly assume the existence of such an admissible metric.


The following foundational result establishes the precompactness of equivariant Gromov--Hausdorff convergence and its compatibility with quotient operations.

\begin{theorem}[\cite{Fukaya1986, FukayaYamaguchi1992}]\label{thm:eGH_compactness}
Let $(X_i,p_i)$ be a sequence of proper geodesic metric spaces converging to $(X,p)$ in the pointed Gromov--Hausdorff sense. For each $i\in\mathbb{N}$, let $G_i$ be a closed subgroup of $\mathrm{Iso}(X_i)$. Then, after passing to a subsequence if necessary, there exists a closed subgroup $G \subseteq \mathrm{Iso}(X)$ such that
$(X_i,p_i,G_i) \xrightarrow{eGH} (X,p,G)$.
Moreover, the corresponding quotient spaces converge in the pointed Gromov--Hausdorff sense:
$(X_i/G_i, \bar{p}_i) \xrightarrow{GH} (X/G, \bar{p})$.
\end{theorem}


\subsection{Nilpotent Lie group}
It is well-known in the regularity theory of Ricci limit spaces that if $Y$ is a Ricci limit space, then its isometry group $\mathrm{Iso}(Y)$ is a Lie group \cite{CheegerColding2000a, ColdingNaber2013}.

In view of the structural results discussed in Section~5, we recall some fundamental properties of nilpotent Lie groups; for a comprehensive treatment, see \cite[Chapter 14]{Liegroup}.

\begin{lemma}\label{lem:nilpotent_commute}
Let $G$ be a nilpotent Lie group with identity component $G_0$. Then any compact subgroup $K \subseteq G$ commutes with $G_0$, i.e., for every $k \in K$ and $g \in G_0$, we have $kg = gk$.
\end{lemma}

\begin{proof}
Let $\mathfrak{g}$ be the Lie algebra of $G_0$, and consider the adjoint representation $\mathrm{Ad}: G \to \mathrm{Aut}(\mathfrak{g})$. Since $K \subseteq G$ is compact, its image $\mathrm{Ad}(K)$ is compact in $\mathrm{GL}(\mathfrak{g})$.

For any $g \in K$, the nilpotency of $G$ implies that the adjoint operator $\mathrm{ad}_g: X \mapsto [g,X]$ is nilpotent. That is, there exists $k>0$ such that the $k$-fold iterated commutator vanishes:
\[
[g,[g,\dots,[g,X]\dots]] = 0,
\]
which is equivalent to $(\mathrm{Ad}(g) - \mathrm{Id})^k = 0$ on $\mathfrak{g}$. Hence, $\mathrm{Ad}(g)$ is unipotent and, in a suitable basis, can be represented as an upper triangular matrix with all diagonal entries equal to $1$.

If $\mathrm{Ad}(g) \neq \mathrm{Id}$, then the subgroup generated by $\mathrm{Ad}(g)$ would form a non-trivial discrete subgroup of a unipotent group, which is necessarily non-compact. This contradicts the compactness of $\mathrm{Ad}(K)$. Therefore, $\mathrm{Ad}(g) = \mathrm{Id}$ for all $g \in K$, which implies that every element of $K$ commutes with all elements of the connected component $G_0$.
\end{proof}

\begin{remark}
While any compact subgroup of a nilpotent Lie group commutes with the identity component, a compact nilpotent group need not be abelian. For instance, the dihedral group $D_{2^k}$ is non-abelian for $k \ge 3$ (e.g., $D_8$) and is nilpotent.
\end{remark}

Recall that a compact subgroup $K \subseteq G$ is called a \emph{maximal compact subgroup} if every compact subgroup of $G$ is conjugate to a subgroup of $K$. For any connected Lie group $G$, a maximal compact subgroup exists and is unique up to conjugation. Moreover, $G$ is homeomorphic to $K \times \mathbb{R}^l$ for some $l \ge 0$.

In the case where $G$ is connected and nilpotent, the structure is even more rigid: every compact subgroup lies in the center $Z(G)$. Consequently, $G$ has a unique maximal compact subgroup, which is a torus $T^k$. This yields the well-known diffeomorphism
\[
G \cong T^k \times \mathbb{R}^l,
\]
so that the quotient $G/T^k$ is simply connected and diffeomorphic to $\mathbb{R}^l$.

\subsection{Pseudo-groups and groupification}

In this section, we introduce the formal concept of a \emph{pseudo-group} and its \emph{groupification}.

\begin{definition}
Let $H$ be a group. A symmetric subset $S \subseteq H$ containing the identity element is called a \emph{pseudo-group} in $H$.
\end{definition}

\begin{remark}
In other contexts, a pseudo-group may be defined as an abstract set equipped with a partial binary operation satisfying certain axioms, without necessarily being contained in a larger group. Throughout this paper, we restrict attention to pseudo-groups embedded in a genuine group $H$.
\end{remark}

\begin{definition}
Let $S_1$ and $S_2$ be pseudo-groups. A map $f \colon S_1 \to S_2$ is a \emph{homomorphism} if, for all $g_1, g_2 \in S_1$ with $g_1 g_2 \in S_1$, we have
$f(g_1) f(g_2) = f(g_1 g_2)$.
\end{definition}

Two pseudo-groups $S_1$ and $S_2$ are \emph{isomorphic} if there exists a bijection $f \colon S_1 \to S_2$ such that both $f$ and $f^{-1}$ are homomorphisms.

We now define the \emph{groupification} of a pseudo-group $S$, denoted $\widehat{S}$. Let $F_S$ be the free group generated by symbols $\{ e_g \}_{g \in S}$, and let $R$ be the normal subgroup of $F_S$ generated by all elements of the form
$e_{g_1} e_{g_2} e_{g_1 g_2}^{-1}$, whenever $g_1, g_2, g_1 g_2 \in S$.
The \emph{groupification} of $S$ is the quotient group
$\widehat{S} \coloneqq F_S / R$.

By construction, there is a natural pseudo-group homomorphism 
\[
\iota \colon S \to \widehat{S}, \qquad g \mapsto [e_g].
\]
For any group $G$ and a homomorphism $\phi \colon S \to G$, the universal property of groupification ensures that $\phi$ extends uniquely to a group homomorphism
$\widehat{\phi} \colon \widehat{S} \to G$
such that the following diagram commutes:
\[
\begin{tikzcd}
S \arrow{r}{\iota} \arrow[swap]{dr}{\phi} & \widehat{S} \arrow{d}{\widehat{\phi}} \\
 & G.
\end{tikzcd}
\]

Let $\langle S \rangle$ be the subgroup of $H$ generated by $S \subset H$. Then there is a canonical surjective homomorphism
\[
\pi \colon \widehat{S} \to \langle S \rangle, \qquad \pi([e_{g_1} \dots e_{g_k}]) = g_1 \dots g_k.
\]
Since $\pi \circ \iota$ is simply the inclusion of $S$ into $H$, the map $\iota$ is injective. This allows us to view $S$ naturally as a subset of its groupification $\widehat{S}$.

\begin{definition}
Let $S$ be a pseudo-group in a group $H$. We say that $S$ is a \emph{defining subset} of $H$ if $S$ generates $H$ and the natural map
\[
\pi \colon \widehat{S} \to H
\]
is an isomorphism.
\end{definition}

In other words, $S$ is a defining subset of $H$ if it contains a generating set for $H$ and all relations in $H$ are locally present in $S$; that is, every global relation in $H$ can be expressed as a product of elementary relations that hold within $S$.

\begin{lemma}\label{lem:global_iso}
Let $S_1$ and $S_2$ be defining subsets of the groups $H_1$ and $H_2$, respectively. If
\[
f \colon S_1 \to S_2
\]
is an isomorphism of pseudo-groups, then $f$ extends uniquely to a group isomorphism
\[
f \colon H_1 \to H_2.
\]
\end{lemma}

We next state a sufficient condition for a pseudo-group to be a defining subset; see \cite{PanWang2021, SantosZamora}.

\begin{lemma}\label{lem:defining_set_ball}
Let $(Y, q)$ be a simply connected proper geodesic metric space, and let $H \subseteq \mathrm{Iso}(Y)$ be a closed subgroup such that 
\[
\mathrm{diam}(Y/H) \le D
\]
for some $D > 0$. Then the set
\[
H(q, 20D) = \{ g \in H \mid d(gq, q) \le 20D \}
\]
is a defining subset of $H$.
\end{lemma}

\begin{proof}[Sketch of proof]
First, note that $H(q, 3D)$ generates $H$ because the quotient diameter is bounded by $D$. Let $\widehat{H}$ denote the groupification of $H(q, 20D)$.  

Equip $H$ with the compact-open topology and we shall show that it induces a natural topology on $\widehat{H}$. Define an open subset of $H$,
\[
K \coloneqq \{ g \in H \mid d(gq, q) < 5D \},
\] 
and let $\{ U_\lambda \}_{\lambda \in \Lambda}$ be the collection of all open subsets of $K$. Using left translations on $\widehat{H}$ and the natural inclusion $\iota \colon K \hookrightarrow \widehat{H}$, we obtain a topological basis for $\widehat{H}$:
\[
\{ \hat{g}\iota(U_\lambda) \mid \hat{g} \in \widehat{H}, \, \lambda \in \Lambda \}.
\]
Since the topology of $H$ is generated by left $H$-translations of $K$, the map $\iota \colon K \to \widehat{H}$ is an embedding.

Next, let $\bar{B}_{10D}(q)$ denote the closed ball of radius $10D$ centered at $q$. Consider the glued space
\[
\widehat{Y} \coloneqq \widehat{H} \times_{H(q, 20D)} \bar{B}_{10D}(q) 
\coloneqq 
\big(\widehat{H} \times \bar{B}_{10D}(q)\big) \Big/ \sim,
\]
where the equivalence relation $\sim$ is defined by
\[
(\hat{g}\iota(g), x) \sim (\hat{g}, gx) \quad \text{for all } g \in H(q, 20D), \, x \in \bar{B}_{10D}(q) \text{ with } gx \in \bar{B}_{10D}(q).
\]
Denote equivalence classes by $[\hat{g}, x]$.

To see that $\widehat{Y}$ is connected, observe that $\bar{B}_{10D}(q)$ is path-connected and each set $[\hat{g}, \bar{B}_{10D}(q)]$ lies in a single component. Moreover, for any $g \in H(q, 20D)$, the sets $[\hat{g}, \bar{B}_{10D}(q)]$ and $[\hat{g}\iota(g), \bar{B}_{10D}(q)]$ intersect non-trivially since $g\bar{B}_{10D}(q) \cap \bar{B}_{10D}(q) \neq \emptyset$. Since $\iota(H(q, 20D))$ generates $\widehat{H}$, the entire space $\widehat{Y}$ is connected.

Define the evaluation map
\[
\Psi \colon \widehat{Y} \to Y, \qquad \Psi([\hat{g}, x]) = \pi(\hat{g})x,
\]
where $\pi \colon \widehat{H} \to H$ is the natural projection. The map $\Psi$ is well-defined by the definition of $\sim$ and surjective because $\mathrm{diam}(Y/H) \le D$ and $\pi$ is surjective. Using arguments from \cite{PanWang2021, Zamora2020}, $\Psi$ is a local homeomorphism. Since $Y$ is simply connected and $\widehat{Y}$ is connected, $\Psi$ must be a global homeomorphism.

Finally, for any $\hat{g} \in \ker(\pi)$ and any $x \in \bar{B}_{10D}(q)$, we have
\[
\Psi([\hat{g}, x]) = \pi(\hat{g})x = x = \Psi([\hat{e}, x]),
\] 
where $\hat{e} \in \hat{H}$ is the identity element.
Since $\Psi$ is a homeomorphism, it follows that $[\hat{g}, x] = [\hat{e}, x]$, which implies $\hat{g} = \hat{e}$ by \cite[Lemma 5.2]{PanWang2021}. Therefore, $\pi$ is injective and $H(q, 20D)$ is indeed a defining subset of $H$.
\end{proof}

\begin{lemma}\label{lem:defining_quotient}
Let $S$ be a defining subset of a group $H$, and let $H' \trianglelefteq H$ be a normal subgroup. If $S \cap H'$ generates $H'$, then the image 
\[
\pi(S) \subseteq H/H'
\] 
under the canonical projection $\pi \colon H \to H/H'$ is a defining subset of the quotient group $H/H'$.
\end{lemma}

\begin{proof}
Since $S$ is symmetric and generates $H$, it follows immediately that $\pi(S)$ is symmetric and generates $H/H'$. To verify that $\pi(S)$ is a defining subset, we need to show that the natural map from the groupification of $\pi(S)$ to $H/H'$ is an isomorphism.

Let $\bar{g}_1, \dots, \bar{g}_k \in \pi(S)$ satisfy
\[
\bar{g}_1 \bar{g}_2 \dots \bar{g}_k = \bar{e} \in H/H'.
\]
We must show that the corresponding word in the free group,
\[
e_{\bar{g}_1} e_{\bar{g}_2} \dots e_{\bar{g}_k},
\]
lies in the normal subgroup $R_{\pi(S)} \trianglelefteq F_{\pi(S)}$ generated by the local relations
\[
e_{h_1} e_{h_2} e_{h_1 h_2}^{-1}, \quad h_1, h_2, h_1 h_2 \in \pi(S).
\]

For each $j$ with $1 \le j \le k$, choose a lift $g_j \in S$ with $\pi(g_j) = \bar{g}_j$. Then
\[
g_1 g_2 \dots g_k \in H'.
\]
By the assumption that $S \cap H'$ generates $H'$, there exist elements $g_{k+1}, \dots, g_{k+m} \in S \cap H'$ such that
\[
g_1 g_2 \dots g_k \, g_{k+1} \dots g_{k+m} = e \in H.
\]

Since $S$ is a defining subset of $H$, the corresponding word in $F_S$ lies in the normal subgroup of relations $R_S$:
\begin{equation}\label{eq:defining_lift}
e_{g_1} e_{g_2} \dots e_{g_{k+m}} \in R_S.
\end{equation}

Applying the projection $\pi$ to \eqref{eq:defining_lift}, all elements in $S \cap H'$ map to the identity in $H/H'$, leaving
\[
e_{\bar{g}_1} e_{\bar{g}_2} \dots e_{\bar{g}_k} \in R_{\pi(S)}.
\]

Hence, every relation in $H/H'$ is captured locally in $\pi(S)$, and we conclude that $\pi(S)$ is a defining subset of $H/H'$.
\end{proof}

\subsection{Gromov--Hausdorff convergence and fundamental group}

The study of the topology of general Ricci limit spaces was initiated by Sormani and Wei, who showed that every Ricci limit space admits a universal cover \cite{SormaniWei2001, SormaniWei2004}. Their construction provided a covering space, but it did not immediately guarantee that this cover is simply connected in the classical sense.

Subsequent work clarified the situation: it is now known that any Ricci limit space—and more generally, any $\mathrm{RCD}(K,N)$ space—is \emph{semi-locally simply connected} \cite{MondinoWei2019,PanWei2019, PanWang2021, Wang2021, Wang2022}. This ensures that classical covering space theory applies, so the universal cover is indeed simply connected.

In this section, we investigate the behavior of the fundamental group under pointed Gromov--Hausdorff convergence $(X_i, p_i) \xrightarrow{GH} (X, p)$,
assuming that each $X_i$ and the limit space $X$ are semi-locally simply connected proper geodesic metric spaces. Our main goal is to relate the fundamental groups $\pi_1(X_i, p_i)$ to $\pi_1(X, p)$.

Let $\gamma$ be a loop based at $p_i \in X_i$. We say that $\gamma$ is \emph{$\epsilon$-contractible} if it is homotopic to a finite concatenation of loops of the form
\[
\alpha_j \cdot \beta_j \cdot \alpha_j^{-1},
\]
where for each $j$:
\begin{itemize}
    \item $\beta_j$ is a loop contained in a ball of radius $\epsilon$;
    \item $\alpha_j$ is a path from $p$ to $\beta(0)$.
\end{itemize}

\begin{lemma}\label{lem:limit_simply_connected}
Let $\{(X_i, p_i)\}_{i \in \mathbb{N}}$ and $(X, p)$ be proper geodesic metric spaces. Suppose there exist a constant $\delta>0$ and a sequence $\epsilon_i \to 0$ such that:
\begin{enumerate}
    \item Any loop in $X$ contained in a ball of radius $\delta$ is contractible in $X$;
    \item Every loop in $X_i$ is $\epsilon_i$-contractible.
\end{enumerate}
If there exist global maps $f_i \colon X_i \to X$ satisfying
\begin{itemize}
    \item $f_i$ is an $\epsilon_i$-GHA on every unit ball,
    \item $f_i$ is $\epsilon_i$-injective,
    \item $f_i(X_i)$ is $\epsilon_i$-dense in $X$,
\end{itemize}
then $X$ is simply connected.
\end{lemma}

\begin{proof}[Sketch of proof]
Let $\gamma$ be an arbitrary loop in $X$ based at $p$. Using the global map $f_i$, lift $\gamma$ to a sequence of loops $\gamma_i$ in $X_i$ that converge pointwise to $\gamma$ as $i \to \infty$.  

By assumption, each $\gamma_i$ is $\epsilon_i$-contractible: it is homotopic to a finite concatenation of loops contained in $\epsilon_i$-balls in $X_i$. By \cite{PanWei2019,Wang2021}, we obtain a corresponding decomposition of $\gamma$ into loops contained in $2\epsilon_i$-balls in $X$.

For sufficiently large $i$, $2\epsilon_i < \delta$, so each small loop is null-homotopic in $X$ by the first hypothesis. Concatenating these null-homotopies shows that $\gamma$ itself is contractible. Since $\gamma$ was arbitrary, $X$ is simply connected.
\end{proof}

\subsection{Slice theorem}

Let $Y$ be a completely regular topological space and $G$ a Lie group. We say that $Y$ is a \emph{$G$-space} if $G$ acts on $Y$ by homeomorphisms.  

For any point $y \in Y$, the \emph{isotropy group} (or stabilizer) at $y$ is
\[
G_y \coloneqq \{ g \in G \mid gy = y \}.
\]

\begin{definition}[$G_y$-invariant subset and twisted product]
A subset $S \subseteq Y$ containing $y$ is \emph{$G_y$-invariant} if $G_y \cdot S = S$.

For such a subset $S$, the \emph{twisted product} (or associated bundle) is defined as
\[
G \times_{G_y} S \coloneqq (G \times S)/\sim,
\]
where the equivalence relation $\sim$ is
\[
(g, s) \sim (g h^{-1}, h s), \quad \forall g \in G, \, h \in G_y, \, s \in S.
\]

The twisted product carries a natural left $G$-action given by
\[
g \cdot [g', s] \coloneqq [g g', s], \quad g, g' \in G, \, s \in S.
\]
\end{definition}

\begin{definition}[Slice at a point]
A subset $S \subseteq Y$ is called a \emph{slice} at $y \in Y$ if:
\begin{enumerate}
    \item $y \in S$ and $S$ is $G_y$-invariant;
    \item The set $G \cdot S$ is an open neighborhood of $y$ in $Y$, and the map
    \[
    \Phi \colon G \times_{G_y} S \longrightarrow G \cdot S, \quad [g, s] \mapsto g s
    \]
    is a $G$-equivariant homeomorphism.
\end{enumerate}
\end{definition}

\begin{remark}
The existence of a slice provides a local model for the quotient space. Specifically, there are canonical homeomorphisms
\[
(G \times_{G_y} S)/G \;\cong\; S/G_y \;\cong\; (G \cdot S)/G,
\]
so that $S/G_y$ serves as a neighborhood of the orbit $[y]$ in $Y/G$.  

Moreover, for any $g \in G$, $gS \cap S \neq \emptyset \quad \Longleftrightarrow \quad g \in G_y$.
In particular, for any $x \in S$, the isotropy subgroup satisfies $G_x \subseteq G_y$.
\end{remark}

\medskip

The existence of slices for proper actions of Lie groups is guaranteed by the classical result of Palais \cite{Palais1961}:

\begin{theorem}[Slice Theorem]\label{thm:topo_slice}
Let $G$ be a Lie group acting on a $G$-space $Y$, and let $y \in Y$. The following are equivalent:
\begin{enumerate}
    \item The isotropy group $G_y$ is compact, and there exists a slice at $y$.
    \item There exists an open neighborhood $U$ of $y$ such that
    \[
    G_U \coloneqq \{ g \in G \mid gU \cap U \neq \emptyset \}
    \]
    has compact closure in $G$.
\end{enumerate}
\end{theorem}

\noindent In the setting of this paper, we are concerned with closed subgroups of the isometry group. For any closed isometric action of a Lie group $G$ on a proper geodesic metric space $X$ (including all Ricci limit spaces), condition (2) is automatically satisfied. Hence, a slice exists at every point $x \in X$ for the action of $G$.

The existence of slices also guarantees the path-lifting property: any continuous path in the quotient $X/G$ can be lifted to a continuous path in $X$. While the lift of a loop is generally a path that may not close, if the loop lies entirely within the image of a slice at its base point, the lift can be chosen to be a loop.  

Furthermore, for any $r>0$, the metric ball $B_r(x)$ is $G_x$-invariant. If $S$ is a slice at $x$, then $S \cap B_r(x)$ is also a slice. 
    
\begin{lemma}\label{lem:quotient_slsc}
Let $X$ be a semi-locally simply connected proper geodesic metric space, and let $G \subseteq \mathrm{Iso}(X)$ be a closed Lie subgroup. Then the quotient space $X/G$ is semi-locally simply connected.
\end{lemma}

\begin{proof}
Since $G$ acts isometrically on the proper geodesic space $X$, a slice exists at every point $p \in X$ by Theorem~\ref{thm:topo_slice}.

Fix $p \in X$ and let $\bar{p} \coloneqq \pi(p) \in X/G$ be its image under the quotient map $\pi \colon X \to X/G$. By the semi-local simple connectivity of $X$, there exists $r>0$ such that any loop contained in $B_r(p) \subset X$ is null-homotopic. Let $S$ be a slice at $p$ with $S \subseteq B_r(p)$. Then
\[
\pi(S) = S/G_p
\]
is an open neighborhood of $\bar{p}$ in $X/G$, so there exists $r' > 0$ such that
\[
B_{r'}(\bar{p}) \subset \pi(S).
\]

Now, consider any loop $\gamma \subset B_{r'}(\bar{p})$ based at $\bar{p}$. By the path-lifting property of slices, $\gamma$ lifts uniquely to a loop $\widetilde{\gamma} \subset S \subset B_r(p)$ in $X$. By the choice of $r$, $\widetilde{\gamma}$ is null-homotopic in $X$. Projecting this null-homotopy via $\pi$ yields a null-homotopy of $\gamma$ in $X/G$.

Since $\gamma$ was arbitrary, every loop in $B_{r'}(\bar{p})$ is contractible, which proves that $X/G$ is semi-locally simply connected.
\end{proof}

\section{Normal subgroups and convergence}

In this section, we prove Theorem~\ref{thm:normal_subgroups}. We begin with two preliminary lemmas that describe the behavior of normal subgroups and quotient spaces.

\begin{lemma}\label{lem:quotient_generation}
Let $(X,p)$ be a proper geodesic space, and let $G \subseteq \mathrm{Iso}(X)$ be a closed subgroup. Suppose $G' \trianglelefteq G$ is a closed normal subgroup. Then $G/G'$ acts isometrically on $X/G'$. 

Moreover, for any $r > 0$, $G/G'(\bar{p}, r) = \pi(G(p, r))$,
where $\pi: G \to G/G'$ is the quotient map and $\bar{p}$ is the image of $p$ in $X/G'$. In particular, for $r = 0$, $\pi(G_p) = (G/G')_{\bar{p}}$.
\end{lemma}

\begin{proof}
For $x, y \in X$ and $g \in G$, let $\pi(x), \pi(y) \in X/G'$ and $\pi(g) \in G/G'$ denote their images under the quotient map. Define the action of $G/G'$ on $X/G'$ by
\[
\pi(g) \cdot \pi(x) := \pi(gx).
\]
This action is well-defined because $G'$ is normal in $G$.

Recall that the quotient metric on $X/G'$ is given by
\[
d(\pi(x), \pi(y)) = \inf_{h \in G'} d(x, hy).
\]
Since $G'$ is closed and $X$ is proper, the infimum is achieved as a minimum:
\[
d(\pi(x), \pi(y)) = \min_{h \in G'} d(x, hy).
\]

For any $g \in G$, we then have
\[
d(\pi(gx), \pi(gy)) = \min_{h \in G'} d(gx, hgy).
\]
Since $G'$ is normal in $G$, for each $h \in G'$, $h' := g^{-1} h g \in G'$. Hence,
\[
d(\pi(gx), \pi(gy)) = \min_{h' \in G'} d(gx, g h' y) = \min_{h' \in G'} d(x, h' y),
\]
which shows that
\[
d(\pi(g) \cdot \pi(x), \pi(g) \cdot \pi(y)) = d(\pi(x), \pi(y)).
\]
Therefore, $G/G'$ acts isometrically on $X/G'$.

Next, we show that for any $r>0$, $G/G'(\bar{p}, r) = \pi(G(p, r))$.

Take any $g \in G(p, r)$. By the definition of the quotient metric,
\[
d(\pi(g)\bar{p}, \bar{p}) \le d(gp, p) \le r.
\] 
Therefore $\pi(g) \in G/G'(\bar{p}, r)$ and $\pi(G(p, r)) \subseteq G/G'(\bar{p}, r)$.

Conversely, let $\bar{g} \in G/G'(\bar{p}, r)$. By the quotient metric definition and the properness of $X$, there exists a representative $g \in G$ of $\bar{g}$ and some $h \in G'$ such that $d(gp, hp) \le r$.
Then $h^{-1} g \in G(p, r)$, and since $\bar{g} = \pi(g) = \pi(h^{-1} g)$ in the quotient, we have $
\bar{g} \in \pi(G(p, r))$, establishing the reverse inclusion $G/G'(\bar{p}, r) \subseteq \pi(G(p, r))$.
\end{proof}

\begin{lemma}\label{lem:quotient_eGH}
Let $\{(X_i, p_i, G_i)\}_{i \in \mathbb{N}}$ be a sequence of proper geodesic metric spaces, with $G_i \subseteq \mathrm{Iso}(X_i)$ closed subgroups. Suppose $G_i' \trianglelefteq G_i$ is a closed normal subgroup for each $i$, and assume
\[
(X_i, p_i, G_i, G_i') \xrightarrow{eGH} (X, p, G, G').
\]
Then $G'$ is a normal subgroup of $G$, and the following diagram commutes:
\[
\begin{CD}
(X_i, p_i, G_i, G_i') @>eGH>> (X, p, G, G')\\
@VV\pi V @VV\pi V\\
(X_i/G_i', \bar{p}_i, G_i/G_i') @>eGH>> (X/G', \bar{p}, G/G').
\end{CD}
\]
\end{lemma}

\begin{proof}
We first show that $G'$ is normal in $G$. Let $g \in G'$ and $h \in G$, and choose $r>0$ sufficiently large so that $g \in G'(p,r)$ and $h \in G(p,r)$. Take sequences $g_i \in G_i'(p_i,r)$ and $h_i \in G_i(p_i,r)$ converging to $g$ and $h$, respectively. Since $G_i'$ is normal in $G_i$, we have
\[
h_i^{-1} g_i h_i \in G_i'(p_i,3r).
\]
Passing to the limit gives
\[
h^{-1} g h \in G'(p,3r),
\]
so $G'$ is normal in $G$.

Next, we show that 
\[
(X_i/G_i', \bar{p}_i, G_i/G_i') \xrightarrow{eGH} (X/G', \bar{p}, G/G').
\]

Assume that $(f_i, \phi_i, \psi_i)$ is an $\epsilon_i$-eGHA from $(X_i, p_i, G_i, G_i')$ to $(X, p, G, G')$. We construct an $\epsilon_i$-eGHA $(\bar{f}_i, \bar{\phi}_i, \bar{\psi}_i)$ from $(X_i/G_i', \bar{p}_i, G_i/G_i')$ to $(X/G', \bar{p}, G/G')$ as follows.

For $\bar{x}_i \in B_{1/\epsilon_i}(\bar{p}_i)$, choose a preimage $x_i \in B_{1/\epsilon_i}(p_i)$ and define
\[
\bar{f}_i(\bar{x}_i) := \pi(f_i(x_i)) \in X/G'.
\]
By the properties of $(f_i, \phi_i, \psi_i)$, $\bar{f}_i$ is well-defined up to an $\epsilon_i$ error and satisfies the $\epsilon_i$-GHA conditions.

Similarly, define
\[
\bar{\phi}_i : G_i/G_i'(\bar{p}_i, 1/\epsilon_i) \to G/G'(\bar{p}, 1/\epsilon_i)
\]
by choosing a representative $h_i \in G_i(p_i, 1/\epsilon_i)$ for $\bar{h}_i \in G_i/G_i'(\bar{p}_i, 1/\epsilon_i)$ and setting
\[
\bar{\phi}_i(\bar{h}_i) := \pi(\phi_i(h_i)).
\]
Such $h_i \in G_i(p_i, 1/\epsilon_i)$ exists by Lemma~\ref{lem:quotient_generation}. 
The map $\bar{\psi}_i$ is defined analogously. Each map is well-defined up to an $\epsilon_i$ error due to the properties of $(f_i, \phi_i, \psi_i)$.

To verify the equivariance property, let $\bar{x}_i \in B_{1/\epsilon_i}(\bar{p}_i)$ and $\bar{h}_i \in G_i/G_i'(\bar{p}_i, 1/\epsilon_i)$. Choosing preimages $x_i$ and $h_i$ as above, we have
\[
\bar{\phi}_i(\bar{h}_i)\bar{f}_i(\bar{x}_i) = \pi(\phi_i(h_i) f_i(x_i)), \quad
\bar{f}_i(\bar{h}_i\bar{x}_i) = \pi(f_i(h_i x_i)).
\]
Since $(f_i, \phi_i, \psi_i)$ is an $\epsilon_i$-eGHA, 
\[
d(\phi_i(h_i) f_i(x_i), f_i(h_i x_i)) \le \epsilon_i.
\]
Passing to the quotient gives
\[
d\big(\bar{\phi}_i(\bar{h}_i)\bar{f}_i(\bar{x}_i), \bar{f}_i(\bar{h}_i\bar{x}_i)\big) \le \epsilon_i,
\]
so $\bar{\phi}_i$ is an $\epsilon_i$-eGHA.  
\end{proof}

\begin{remark}\label{rem:non_effective}
In Lemma~\ref{lem:quotient_generation}, the quotient group $G/G'$ may not act effectively on $X/G'$; that is, some non-trivial elements may act trivially. For instance, if $X = \mathbb{R}$, $G = \mathrm{Iso}(\mathbb{R})$, and $G'$ is the subgroup of translations, then $X/G'$ is a single point, while $G/G' \cong \mathbb{Z}_2$ acts trivially.

In Lemma~\ref{lem:quotient_eGH}, $(\bar{f}_i, \bar{\phi}_i, \bar{\psi}_i)$ satisfies the definition of an eGHA without assuming that the group actions are effective.
\end{remark}

We now prove Theorem~\ref{thm:normal_subgroups}. We remark that Theorem~\ref{thm:normal_subgroups} was previously established by Fukaya and Yamaguchi under the more restrictive assumption that each $G_i$ is discrete and free \cite{FukayaYamaguchi1994}. Our proof generalizes their results, building upon techniques developed in \cite{FukayaYamaguchi1992, FukayaYamaguchi1994, PanWang2021, SantosZamora}.

\begin{theorem}[Theorem~\ref{thm:normal_subgroups}]\label{Ide}
Let $\{(X_i,p_i,G_i)\}_{i\in\mathbb{N}}$ be a sequence of simply connected proper geodesic metric spaces, where each $G_i\subseteq \mathrm{Iso}(X_i)$ is a closed subgroup. Suppose that
\[
(X_i,p_i,G_i)\xrightarrow{eGH}(X,p,G),
\]
and assume that there exists a closed normal subgroup $G'\trianglelefteq G$ satisfying:
\begin{enumerate}
\item $G/G'$ is discrete;
\item $G'$ is generated by $G'(p,R)$ for some $R>0$;
\item $\mathrm{diam}(X_i/G_i)\le D$ for some fixed constant $D>0$.
\end{enumerate}
Then there exists a sequence of normal subgroups $G_i'\trianglelefteq G_i$ such that:
\begin{enumerate}
\item $G_i'$ converges to $G'$ in the equivariant sense;
\item $G_i/G_i' \cong G/G'$ for all sufficiently large $i$;
\item $G_i'$ is generated by $G_i'(p_i,R+\epsilon_i)$, where $\epsilon_i\to 0$;
\item $G/G'$ is finitely presented.
\end{enumerate}
\end{theorem}

Without loss of generality, we may assume $D \ge R$. Let $S_i(20D)$ denote the subset of elements in $G_i(p_i, 20D)$ that are $\epsilon_i$-close to elements of $G'$.

Since $G/G'$ is discrete, there exists a constant $\epsilon > 0$ such that for any $g \in G(p, 20D)$ and $h \in G'(p, 20D)$, if $g$ is $\epsilon$-close to $h$—meaning $d(gx, hx) < \epsilon$ for all $x \in B_{1/\epsilon}(p)$—then $g$ must belong to $G'$.

Consequently, for sufficiently large $i$, if $g_i \in G_i(p_i, 20D)$ is $\epsilon/2$-close to an element $h_i \in S_i(20D)$, then $g_i$ also lies in $S_i(20D)$. In particular, any element of $G_i$ that is $\epsilon/2$-close to the identity is contained in $S_i(20D)$.

It follows that $S_i(20D)$ is a closed subset of $G_i(p_i, 20D)$ for all sufficiently large $i$. Moreover, since $S_i(20D)$ contains a neighborhood of the identity, the subgroup
\[
G_i' \coloneqq \langle S_i(20D) \rangle \subset G_i
\]
is a closed subgroup of $G_i$.

\begin{lemma}\label{lem:normality_Gi}
For all sufficiently large $i$, the subgroup $G_i' \subset G_i$ is normal.
\end{lemma}

\begin{proof}
Since $G_i$ is generated by $G_i(p_i, 20D)$, it suffices to show that
\[
g G_i' g^{-1} = G_i' \quad \text{for all } g \in G_i(p_i, 20D).
\] 
By the definition $G_i' = \langle S_i(20D) \rangle$, it is enough to verify that
\[
g S_i(20D) g^{-1} \subseteq G_i' \quad \text{for all } g \in G_i(p_i, 20D).
\]

Suppose, for a contradiction, that there exist sequences $g_i \in G_i(p_i, 20D)$ and $h_i \in S_i(20D)$ such that 
\[
g_i h_i g_i^{-1} \notin G_i'.
\] 
Passing to a subsequence if necessary, we may assume these sequences converge in the equivariant Gromov--Hausdorff sense:
\[
g_i \to g \in G(p, 20D), \quad h_i \to h \in G'(p, 20D), \quad g_i h_i g_i^{-1} \to ghg^{-1} \in G'(p, 60D).
\]
Since $G'$ is normal in $G$, the limit element $ghg^{-1}$ lies in $G'$.

Because $G'$ is generated by $G'(p, R)$ with $R \le D$, we can write
\[
ghg^{-1} = t_1 t_2 \cdots t_k,
\]
where each $t_j \in G'(p, D)$. For each $t_j$, choose sequences $t_{ij} \in S_i(20D)$ converging to $t_j$. Then, for sufficiently large $i$, the product
\[
t_{i1} t_{i2} \cdots t_{ik} \in G_i'
\]
approximates $g_i h_i g_i^{-1}$. More precisely, we can write
\[
g_i h_i g_i^{-1} = (t_{i1} t_{i2} \cdots t_{ik}) s_i,
\]
where $s_i \in G_i$ converges to the identity. By construction, any element sufficiently close to the identity lies in $S_i(20D) \subseteq G_i'$. Therefore, $g_i h_i g_i^{-1} \in G_i'$, contradicting the initial assumption.
\end{proof}

By Lemma~\ref{lem:quotient_eGH}, we obtain the following commutative diagram in the sense of equivariant Gromov--Hausdorff convergence:
\begin{center}
	$\begin{CD}
		(X_i, p_i, G_i, G_i') @>eGH>> (X, p, G, G'')\\
		@VV\pi V @VV\pi V\\
		(X_i/G_i', \bar{p}_i, G_i/G_i') @>eGH>> (X/G'', \bar{p}, G/G'')
	\end{CD}$
\end{center}

Recall that $G_i' \coloneqq \langle S_i(20D) \rangle$, and define the subset of $G_i'$ consisting of elements that move $p_i$ by at most $20D$:
\[
G_i'(p_i, 20D) \coloneqq \{ g \in G_i' \mid d(gp_i, p_i) \le 20D \}.
\]
While it is immediate that $S_i(20D) \subseteq G_i'(p_i, 20D)$, equality is not yet established.

By construction, the limit group $G''$ contains $G'(p, 20D)$. Since $G'$ is generated by $G'(p, R)$ and $R \le D$, it follows that $G' \subseteq G''$. Consequently, the quotient $G/G''$ is discrete. We will verify at the end of this section that $G' = G''$ and $S_i(20D) = G_i'(p_i, 20D)$.

Because $G/G''$ is discrete, we may assume, without loss of generality, that there are no elements in the quotient moving the basepoint exactly a distance of $20D$:
\[
\{ \bar{g} \in G/G'' \mid d(\bar{p}, \bar{g}\bar{p}) = 20D \} = \emptyset.
\]
If this condition fails, one can replace $20D$ by $20D + \epsilon$ for a sufficiently small $\epsilon > 0$ so that the corresponding set at radius $20D + \epsilon$ is empty. Therefore, for sufficiently large $i$,
\[
\{ \bar{g}_i \in G_i/G_i' \mid d(\bar{p}_i, \bar{g}_i\bar{p}_i) = 20D \} = \emptyset.
\]

We now show that the equivariant Gromov--Hausdorff approximation induces a pseudo-group isomorphism between the quotient pseudo-groups.

\begin{lemma}\label{lem:pseudo_iso}
For all sufficiently large $i$, the equivariant Gromov--Hausdorff approximation
\[
\bar{\phi}_i \colon G_i/G_i'(\bar{p}_i, 20D) \longrightarrow G/G''(\bar{p}, 20D)
\]
is a pseudo-group isomorphism. In particular, the set $G_i/G_i'(\bar{p}_i, 20D)$ is discrete.
\end{lemma}
\begin{proof}
Let $(f_i, \phi_i, \psi_i)$ be an $\epsilon_i$-eGHA from $(X_i, p_i, G_i)$ to $(X, p, G)$.  
By the proof of Lemma~\ref{lem:quotient_eGH}, we can construct an $\epsilon_i$-eGHA
\[
(\bar{f}_i, \bar{\phi}_i, \bar{\psi}_i) : (X_i/G_i', \bar{p}_i, G_i/G_i') \longrightarrow (X/G'', \bar{p}, G'').
\]

We first show that
\[
\bar{\phi}_i \colon G_i/G_i'(\bar{p}_i, 20D) \longrightarrow G/G''(\bar{p}, 20D)
\]
is a pseudo-group homomorphism.  
Take any $\bar{g}_1, \bar{g}_2 \in G_i/G_i'(\bar{p}_i, 20D)$.  
By Lemma~\ref{lem:quotient_generation}, there exist preimages $g_1, g_2 \in G_i(p_i, 20D)$.  

Since $\phi_i$ is an almost homomorphism and $G_i'$ contains a $\frac{\epsilon}{2}$-neighborhood of $e \in G_i$, we have
\[
\phi_i(g_1 g_2)(\phi_i(g_1) \phi_i(g_2))^{-1} \in G_i'.
\]  
Passing to the quotient via $\pi: G_i \to G_i/G_i'$, it follows that
\[
\bar{\phi}_i(\bar{g}_1)\, \bar{\phi}_i(\bar{g}_2) 
= \pi(\phi_i(g_1) \phi_i(g_2)) 
= \pi(\phi_i(g_1 g_2)) 
= \bar{\phi}_i(\bar{g}_1 \bar{g}_2).
\]  
Hence $\bar{\phi}_i$ is a pseudo-group homomorphism.

Similarly, one can show that
\[
\bar{\psi}_i \colon G/G''(\bar{p}, 20D) \longrightarrow G_i/G_i'(\bar{p}_i, 20D)
\]
is also a pseudo-group homomorphism.

Next, we show that 
\[
\bar{\phi}_i \circ \bar{\psi}_i : G_i/G_i'(\bar{p}_i, 20D) \longrightarrow G_i/G_i'(\bar{p}_i, 20D)
\] 
is the identity map. Indeed, for any $\bar{g}_i \in G_i/G_i'(\bar{p}_i, 20D)$ with a preimage $g_i \in G_i(p_i, 20D)$, $
\phi_i \circ \psi_i (g_i)$ is $\epsilon_i$-close to $g_i$.
By a similar argument, it follows that 
\[
\bar{\phi}_i \circ \bar{\psi}_i (\bar{g}_i) = \bar{g}_i.
\]  

Similarly, one can show that 
\[
\bar{\psi}_i \circ \bar{\phi}_i : G/G''(\bar{p}, 20D) \longrightarrow G/G''(\bar{p}, 20D)
\] 
is the identity map. Therefore, $\bar{\phi}_i$ is a pseudo-group isomorphism.
\end{proof}




\begin{lemma}\label{lem:isom_quotients}
For all sufficiently large $i$, $\bar{\phi}_i$ induces a group isomorphism
\[
G_i/G_i' \cong G/G''.
\]
\end{lemma}

\begin{proof}
By Lemmas~\ref{lem:pseudo_iso} and~\ref{lem:global_iso}, it suffices to show that 
$G_i/G_i'(\bar{p}_i, 20D)$ and $G/G''(\bar{p}, 20D)$ are defining subsets. 

By Lemmas~\ref{lem:defining_set_ball} and~\ref{lem:defining_quotient}, $G_i/G_i'(\bar{p}_i, 20D)$ is a defining subset of $G_i/G_i'$.

We next show that $G/G''(\bar{p}, 20D)$ is also a defining subset. 
Since it is unknown whether $X$ is simply connected, we cannot directly apply Lemma~\ref{lem:defining_set_ball} to $X$ and $G$. 
Instead, we show that the defining property passes to the limit.

The diameter of $(X/G'')/(G/G'') \cong X/G$ is at most $D$, so $G/G''(\bar{p}, 20D)$ generates $G/G''$. 
Let $\pi$ denote the natural surjective homomorphism from the groupification of $G/G''(\bar{p}, 20D)$ onto $G/G''$. 
We only need to show that $\pi$ is injective.

Suppose, for contradiction, that $\ker(\pi)$ contains a non-trivial element represented by the word
\[
[e_{\bar{g}_1} e_{\bar{g}_2} \cdots e_{\bar{g}_k}],
\] 
where each $\bar{g}_j \in G/G''(\bar{p}, 20D)$ and $\bar{g}_1 \bar{g}_2 \cdots \bar{g}_k = \bar{e}$ in $G/G''$.  
For each $j$, take a preimage $g_j \in G(p, 20D)$ of $\bar{g}_j$. Then 
\[
g_1 g_2 \cdots g_k \in G''(p, 20kD).
\] 

Using the pseudo-group isomorphism $\bar{\psi}_i = \bar{\phi}_i^{-1}$, define
\[
\bar{g}_{ij} \coloneqq \bar{\psi}_i(\bar{g}_j) \in G_i/G_i'(\bar{p}_i, 20D)
\]
for $1 \le j \le k$.  
By the construction in Lemma~\ref{lem:quotient_eGH}, we may take the preimage of $\bar{g}_{ij}$ to be 
\[
g_{ij} = \psi_i(g_j) \in G_i(p, 20D).
\]  

By Lemma~\ref{lem:pseudo_iso}, the word
\[
[e_{\bar{g}_{i1}} e_{\bar{g}_{i2}} \cdots e_{\bar{g}_{ik}}]
\] 
is non-trivial in the groupification of $G_i/G_i'(\bar{p}_i, 20D)$.  
Since $G_i/G_i'(\bar{p}_i, 20D)$ is a defining subset of $G_i/G_i'$, it follows that 
\[
\bar{g}_{i1} \bar{g}_{i2} \cdots \bar{g}_{ik} \neq \bar{e} \in G_i/G_i'.
\]  
Hence, their preimages satisfy 
\[
g_{i1} g_{i2} \cdots g_{ik} \notin G_i'.
\]

On the other hand, $g_{i1} g_{i2} \cdots g_{ik}$ converges to $g_1 g_2 \cdots g_k \in G''(p, 20kD)$.  
Since $G_i' \to G''$ and $G_i'$ contains an $\frac{\epsilon}{2}$-neighborhood of the identity, we must have 
\[
g_{i1} g_{i2} \cdots g_{ik} \in G_i',
\] 
a contradiction.  
Thus $\pi$ is injective, and $G/G''(\bar{p}, 20D)$ is a defining subset.  
\end{proof}



\begin{proof}[Proof of Theorem~\ref{Ide}]
By Lemma~\ref{lem:isom_quotients}, we have $G_i' \to G''$ and $G_i/G_i' \cong G/G''$ for all sufficiently large $i$.  
To complete the proof, it remains to show that $G'' = G'$.

First, note that $G' \subseteq G''$, and since both are normal in $G$, we have $G' \trianglelefteq G''$.  
Fix a sufficiently large $i$, consider the composition of maps
\[
G_i(p_i, 20D) \xrightarrow{\phi_i} G \xrightarrow{\pi'} G/G' \xrightarrow{\pi''} G/G'',
\]
where $\pi'$ and $\pi''$ are the canonical quotient homomorphisms.  
Since $\phi_i$ is an approximation and both $G/G'$ and $G/G''$ are discrete, the compositions $\pi' \circ \phi_i$ and $\pi'' \circ \pi' \circ \phi_i$ are pseudo-group homomorphisms (see also the proof of Lemma~\ref{lem:pseudo_iso}).

Because $G_i(p_i, 20D)$ is a defining set for $G_i$, these maps extend uniquely to global homomorphisms
\[
\Phi_i \colon G_i \to G/G' \quad \text{and} \quad \pi'' \circ \Phi_i \colon G_i \to G/G''.
\]  
Moreover, $\Phi_i$ is surjective, since $G$ is generated by $G(p, 3D)$, $G/G'$ is discrete, and $\phi_i$ is $\epsilon_i$-dense.

By construction, $S_i(20D) \subseteq \ker(\pi' \circ \phi_i)$, so the subgroup $G_i' = \langle S_i(20D) \rangle$ lies in the kernel of $\Phi_i$.  
Hence, we obtain well-defined surjective homomorphisms on the quotients:
\[
\bar{\Phi}_i \colon G_i/G_i' \to G/G' \quad \text{and} \quad \pi'' \circ \bar{\Phi}_i \colon G_i/G_i' \to G/G''.
\]

Lemma~\ref{lem:isom_quotients} implies that the composition $\bar{\phi}_i = \pi'' \circ \bar{\Phi}_i$ is an isomorphism.  
For a composition of surjective maps to be an isomorphism, the final map $\pi''$ must itself be an isomorphism.  
Therefore, $\ker(\pi'')$ is trivial, which shows $G' = G''$.

Next, observe that $G/G'$ is finitely presented because $G/G'(\bar{p}, 20D)$ is a finite defining set.

Finally, we verify that $G_i'$ is generated by $G_i'(p_i, R + \epsilon_i)$ for some $\epsilon_i \to 0$.  
Recall that $G_i'$ is initially generated by $G_i'(p_i, 20D)$. Suppose, for contradiction, that there exists $\delta > 0$ and a sequence $g_i \in G_i'(p_i, 20D)$ such that
\[
g_i \notin \langle G_i'(p_i, R + \delta) \rangle
\] 
for all sufficiently large $i$.  
Passing to a subsequence, assume $g_i \to g \in G'$. Since $G'$ is generated by $G'(p, R)$, we can write
\[
g = h_1 h_2 \cdots h_k, \quad h_j \in G'(p, R).
\]

For each $j$, select $h_{ij} \in G_i'(p_i, R + \delta)$ converging to $h_j$. Then
\[
h_{i1} h_{i2} \cdots h_{ik} s_i = g_i,
\] 
where $s_i \in G_i$ converges to the identity.  
By construction, $s_i$ eventually lies in $G_i'(p_i, R + \delta)$ for large $i$, implying
\[
g_i \in \langle G_i'(p_i, R + \delta) \rangle,
\] 
contradicting our assumption.
\end{proof}

\begin{remark}
Since $S_i(20D)$ contains a $\frac{\epsilon}{2}$-neighborhood of the identity, it follows from $G' = G''$ that 
$S_i(20D) = G_i'(p_i, 20D)$.
\end{remark}




\section{Extend a local eGHA to a global map}
In this section, we present the proofs of Theorem~\ref{thm:extend} and Lemma~\ref{lem:quotient_topology}. These results provide the technical framework necessary to establish Theorem~\ref{thm:G_0}.

\begin{theorem}[Theorem~\ref{thm:extend}]\label{glo app}
Let $(X,p,G)$ and $(Y,q,G)$ be proper geodesic metric spaces equipped with isometric $G$-actions. Suppose that for some constants $D>10\epsilon>0$ with $D < \tfrac{1}{10\epsilon}$, there exists an $\epsilon$-Gromov--Hausdorff approximation
\[
f: B_{4D}(p)\to B_{4D}(q)
\]
that is $\epsilon$-almost $G$-equivariant. That is, for any $x\in B_{4D}(p)$ and $g\in G$ such that $gx\in B_{4D}(p)$, we have
\[
d\big(gf(x),f(gx)\big)\le \epsilon.
\]
Assume further that the quotient diameters satisfy $\mathrm{diam}(X/G)\le D$ and $\mathrm{diam}(Y/G)\le D$. Then there exists a global map $f':X\to Y$ such that:
\begin{enumerate}
\item $f'$ is $2\epsilon$-close to $f$ on $B_{4D}(p)$;
\item $f'$ is a $2\epsilon$-Gromov--Hausdorff approximation on every $D$-ball in $X$;
\item $f'$ is $2\epsilon$-almost $G$-equivariant on $X$;
\item $f'$ is $\epsilon$-injective, and the image $f'(X)$ is $3\epsilon$-dense in $Y$.
\end{enumerate}
\end{theorem}

\begin{proof}  
Since $\mathrm{diam}(X/G) \le D$, for any $x \in X$ there exists $g \in G$ such that $d(p, g x) \le D$.  
We define
$f'(x) \coloneqq g^{-1} f(gx)$, 
choosing an arbitrary $g$ satisfying $d(p, gx) \le D$.

We first prove that $f'$ is a $2\epsilon$-GHA on every $D$-ball in $X$.  
For any $x_1, x_2 \in X$ with $d(x_1, x_2) \le D$, we need to show
\begin{equation}\label{e2}
\big| d(f'(x_1), f'(x_2)) - d(x_1, x_2) \big| \le 2\epsilon.
\end{equation}

Let $x_1, x_2 \in X$ with $d(x_1, x_2) \le D$, and choose $g_1, g_2 \in G$ such that $d(p, g_j x_j) \le D$ and $f'(x_j) = g_j^{-1} f(g_j x_j)$ for $j=1,2$.  
We need to show
\begin{equation}\label{e7}
\big| d(g_1^{-1} f(g_1 x_1), g_2^{-1} f(g_2 x_2)) - d(x_1, x_2) \big| \le 2\epsilon.
\end{equation}

The displacement of the basepoint under $g_1 g_2^{-1}$ satisfies
\[
d(g_1 g_2^{-1} p, p) \le d(g_2^{-1} p, x_2) + d(x_2, x_1) + d(x_1, g_1^{-1} p) \le 3D,
\]  
so $g_1 g_2^{-1} \in G(p, 3D)$.  

Also, we have
\begin{equation}\label{eq:displacement}
d(g_1 x_2, p) \le d(g_1 x_2, g_1 x_1) + d(g_1 x_1, p) \le 2D.
\end{equation}

Since $f \colon B_{4D}(p) \to B_{4D}(q)$ is $\epsilon$-almost $G$-equivariant and $(g_1 g_2^{-1}) g_2 x_2 = g_1 x_2 \in B_{2D}(p)$, we have
\[
d\big(g_1 g_2^{-1} f(g_2 x_2), f(g_1 x_2)\big) \le \epsilon.
\]  
Using $f'(x_2) = g_2^{-1} f(g_2 x_2)$ and that $g_1$ is an isometry, it follows that
\begin{equation}\label{eq:well_defined_lift}
d(f'(x_2), g_1^{-1} f(g_1 x_2)) \le \epsilon.
\end{equation}

Moreover, since $g_1 x_1$ and $g_1 x_2$ lie in $B_{2D}(p)$, the GHA property of $f$ gives
\[
|d(f(g_1 x_1), f(g_1 x_2)) - d(g_1 x_1, g_1 x_2)| \le \epsilon.
\]  
Using $f'(x_1) = g_1^{-1} f(g_1 x_1)$ and the fact that $g_1$ is an isometry, we obtain
\begin{equation}\label{eq:local_gh}
|d(f'(x_1), g_1^{-1} f(g_1 x_2)) - d(x_1, x_2)| \le \epsilon.
\end{equation}

Combining \eqref{eq:well_defined_lift} and \eqref{eq:local_gh} via the triangle inequality, we deduce \eqref{e2}:
\[
|d(f'(x_1), f'(x_2)) - d(x_1, x_2)| \le 2\epsilon.
\]

\eqref{e7} implies that the definition of $f'$ is independent of the choice of $g$, up to an error of at most $2\epsilon$.  
More precisely, if $g'$ is another choice with $d(p, g'x) \le D$, then
\[
d\big(g^{-1} f(gx), (g')^{-1} f(g'x)\big) \le 2\epsilon.
\]  
\eqref{e7} also implies that $f'$ is almost $G$-equivariant and $f'$ is $2\epsilon$-close to $f$ on $B_{4D}(p)$. 

Next, we show that $f'$ is $\epsilon$-injective.  
Take $x_1, x_2 \in X$ with $d(f'(x_1), f'(x_2)) < \epsilon$.  
Since $f'$ is almost $G$-equivariant and the quotient diameters are bounded by $D$, we may assume $x_1 \in B_D(p)$ and $f'(x_1) = f(x_1)$.  

There exists $g \in G$ with $g x_2 \in B_D(p)$ and $f'(x_2) = g^{-1} f(g x_2)$. Since $x_1 \in B_D(p)$ and $d(f(x_1),f'(x_2)) < \epsilon$, $f'(x_2) = g^{-1} f(g x_2) \in B_{D+2\epsilon}(q)$. By our assumption $g x_2 \in B_D(p)$, then $f(gx_2) \in B_{D+\epsilon}(q)$. Therefore $g^{-1} \in G(q, 2D+3\epsilon) \subset G(p, 2D+4\epsilon)$.  
Hence $x_2= g^{-1}(gx_2) \in B_{3D+4\epsilon}(p) \subset B_{4D}(p)$. 

Then we may assume $f'(x_2)=f(x_2)$. Since $f$ is a GHA and $d(f(x_1),f(x_2))<\epsilon$, we conclude that $d(x_1, x_2) < 3\epsilon$.  
Thus $f'$ is $\epsilon$-injective.

Finally, we show that $f'(X)$ is $3\epsilon$-dense in $Y$.  
For any $y \in Y$, there exists $g \in G$ such that $d(g y, q) \le D$.  
Choose $x \in B_{2D}(p)$ with $d(f(x), g y) \le \epsilon$. Then
\[
d(f'(g^{-1} x), g^{-1} f(x)) \le 2\epsilon,
\]  
and by the triangle inequality,
\[
d(y, f'(g^{-1} x)) \le d(y, g^{-1} f(x)) + d(g^{-1} f(x), f'(g^{-1} x)) \le \epsilon + 2\epsilon = 3\epsilon.
\]
\end{proof}


\begin{lemma}[Lemma~\ref{lem:quotient_topology}]
Let $X$ be a simply connected proper geodesic metric space, and suppose that $\mathrm{Iso}(X)$ is a Lie group. Let $G \subseteq \mathrm{Iso}(X)$ be a closed subgroup generated by the identity component $G_0$ together with all isotropy subgroups. Then the quotient space $X/G$ is simply connected.
\end{lemma}

\begin{proof}
We first prove that $X/G_0$ is simply connected. 
Let $\gamma \subset X/G_0$ be a loop based at $\pi(p)$.  
Lift $\gamma$ to a path $\tilde{\gamma}$ in $X$ starting at $p$ and ending at $q = g p$ for some $g \in G_0$.  
Since $G_0$ is path-connected, there exists a path $g(t) \in G_0$ with $g(0) = e$ and $g(1) = g^{-1}$.  
Define
\[
\tilde{\sigma}(t) \coloneqq g(t) \tilde{\gamma}(t),
\]  
which is a loop in $X$ based at $p$.  
Because $X$ is simply connected, $\tilde{\sigma}$ is contractible.  
Projecting back to $X/G_0$, we have $\pi(\tilde{\sigma}(t)) = \gamma(t)$, so $\gamma$ is contractible as well.

Then we prove that $X/G$ is simply connected.  
We have $X/G \cong (X/G_0)/(G/G_0)$.  
Note that $G/G_0$ is generated by the isotropy subgroups of the $G/G_0$-action on $X/G_0$.  
To simplify notation, assume $G_0$ is trivial and that $G$ is generated solely by isotropy elements.

Let $\gamma \subset X/G$ be a loop based at $\bar{p}$.  
Lift $\gamma$ to a path $\tilde{\gamma}$ in $X$ with endpoints $p$ and $q = g p$ for some $g \in G$.  
Since $G$ is generated by isotropy elements, write
\[
g = g_1 g_2 \cdots g_l,
\]  
where each $g_j$ fixes some $x_j \in X$.  

For each $j$, let $\tilde{c}_j$ be a path from $g_j g_{j+1} \cdots g_l p$ to $x_j$, and let $g_j^{-1} \tilde{c}_j^{-1}$ be the corresponding path back to $g_{j+1} \cdots g_l p$.  
Form the concatenated path
\[
\tilde{\sigma} \coloneqq \tilde{c}_1 \cdot (g_1^{-1} \tilde{c}_1^{-1}) \cdot \tilde{c}_2 \cdot (g_2^{-1} \tilde{c}_2^{-1}) \cdots \tilde{c}_l \cdot (g_l^{-1} \tilde{c}_l^{-1}) \cdot \tilde{\gamma}.
\]

By construction, $\tilde{\sigma}$ is a loop in $X$ and hence contractible.  
In the quotient $X/G$, each segment $\pi(\tilde{c}_j)$ cancels with $\pi(g_j^{-1} \tilde{c}_j^{-1})$ up to reparametrization and orientation.  
Thus, $\pi(\tilde{\sigma})$ is homotopic to $\pi(\tilde{\gamma}) = \gamma$.  
Since $\pi(\tilde{\sigma})$ is the image of a contractible loop, it follows that $\gamma$ is contractible.  
\end{proof}


We are ready to prove Theorem~\ref{thm:G_0}.
\begin{theorem}[Theorem~\ref{thm:G_0}]
Let $\{(X_i,p_i)\}_{i\in\mathbb{N}}$ be a sequence of simply connected proper geodesic metric spaces. Suppose $G_i\subseteq \mathrm{Iso}(X_i)$ is a subgroup satisfying the uniform diameter bound $\mathrm{diam}(X_i/G_i)\le D$ for some $D>0$. Assume that the sequence converges in the equivariant Gromov--Hausdorff sense:
\[
(X_i,p_i,G_i)\xrightarrow{eGH}(X,p,G).
\]
If $X$ is semi-locally simply connected and $G$ is a Lie group, then the quotient space $X/G_0$ is simply connected.
\end{theorem}

\begin{proof}
By Theorem~\ref{thm:normal_subgroups}, there exists a sequence of normal subgroups $G_i' \trianglelefteq G_i$ such that
\[
G_i' \xrightarrow{eGH} G_0 \quad \text{and} \quad G_i/G_i' \cong G/G_0
\]
for all sufficiently large $i$.  

Let $H_i \subseteq G_i'$ denote the subgroup generated by the identity component $(G_i')_0$ together with all isotropy subgroups of $G_i'$.  
By Lemma~\ref{lem:quotient_topology}, the quotient space $X_i/H_i$ is simply connected.
 
We claim that $\pi_1(X_i/G_i', \bar{p}_i)$ is generated by loops based at $\bar{p}_i$ whose lengths are bounded by $\epsilon_i$, with $\epsilon_i \to 0$ as $i \to \infty$.  
Indeed, $G_0$ is generated by $G_0(p,\epsilon)$ for any $\epsilon > 0$. Then by the third result of Theorem \ref{thm:normal_subgroups} , $G_i'$ is generated by the set $G_i'(p_i, \epsilon_i)$.  
Consequently, the quotient group $G_i'/H_i$ is generated by the corresponding set $(G_i'/H_i)(\hat{p}_i, \epsilon_i)$.  

Passing to the quotient by the kernel of the action, we may assume that $G_i'/H_i$ acts effectively on $X_i/H_i$.  
Since $H_i$ contains $(G_i)_0$ and all isotropy subgroups of $G_i'$, $G_i'/H_i$-action is free and discrete.  
Therefore, because $X_i/H_i$ is simply connected, we have
\[
\pi_1(X_i/G_i', \bar{p}_i) \cong G_i'/H_i,
\]
which is generated by loops of length at most $\epsilon_i$.

Consider the sequence of quotients:
\[
(X_i/G_i', \bar{p}_i, G_i/G_i') \xrightarrow{eGH} (X/G_0, \bar{p}, G/G_0).
\]  
Since $G_i/G_i' \cong G/G_0$, Theorem~\ref{glo app} provides a global map
\[
f_i' \colon X_i/G_i' \to X/G_0
\]  
that is a $2\epsilon_i$-GHA on every $D$-ball. $f_i'$ is almost injective and the image is almost dense, up to an $\epsilon_i$ error.

By Lemma~\ref{lem:quotient_slsc}, since $X$ is semi-locally simply connected, the quotient space $X/G_0$ is also semi-locally simply connected. Since $X/G_0$ admits a cocompact $G/G_0$-action  and $\mathrm{diam}(X/G) \le D$, we can choose $\delta>0$ such that any loop in a $\delta$-ball of $X/G_0$ is contractible. 
Then we can apply Lemma~\ref{lem:limit_simply_connected} and conclude that $X/G_0$ is simply connected.
\end{proof}


\section{Convergence of discrete groups}

In this section, we establish a Margulis-type result, Lemma~\ref{lem:nilpotent_G0}. Notably, our approach does not require any curvature bounds, and, unlike classical settings, the discrete groups considered here are not assumed to act freely. For related results, see \cite{DSZZ2023,KapovitchWilking2011}.

\begin{lemma}[Lemma~\ref{lem:nilpotent_G0}]\label{nil}
Let $\{(X_i,p_i,G_i)\}_{i\in\mathbb{N}}$ be proper geodesic metric spaces, where each $G_i$ is a discrete subgroup of $\mathrm{Iso}(X_i)$. Suppose that
\[
(X_i,p_i,G_i)\xrightarrow{eGH}(X,p,G),
\]
and that the limit group $G$ is a Lie group. Then the identity component $G_0$ is a nilpotent Lie group.
\end{lemma}

Our proof follows the strategy developed in \cite{SantosZamora}, where a similar lemma was established under the additional assumption that $X/G$ is compact. A key ingredient is the structure theorem for approximate groups due to Breuillard, Green, and Tao.

Recall that for a group $H$, a finite symmetric subset $A \subseteq H$ containing the identity is called a \emph{\(k\)-approximate group} if there exists a set $B \subseteq H$ with $|B| \le k$ such that
\[
A^2 \subseteq B A.
\]

Let $H$ be a group. Inductively define $H_1=[H,H]$ and $H_{k+1}=[H_k,H]$ for any $k \in \mathbb{N}$. $H$ is nilpotent of step at most $C$ if $H_C$ is trivial.

\begin{theorem}[Structure of Approximate Groups \cite{BGT}]\label{thm:BGT_structure}
For any \(k \in \mathbb{N}\), there exists a constant \(k' \in \mathbb{N}\), depending only on \(k\), such that the following holds:  

Let \(H\) be a group generated by a finite symmetric set \(S\) containing the identity, and let \(A \subseteq H\) be a \(k\)-approximate group with \(S^{k'} \subseteq A\). Then there exist a finite normal subgroup \(N \trianglelefteq H\) and a subgroup \(H' \subseteq H\) containing \(N\) such that:
\begin{enumerate}[(i)]
    \item The index \(|H:H'| \le C(k)\);
    \item The quotient \(H'/N\) is nilpotent of step at most \(C(k)\) and is generated by at most \(C(k)\) elements;
    \item \(N \subseteq A^4\).
\end{enumerate}
\end{theorem}

We prove Lemma~\ref{nil} through a sequence of preparatory results.  
For the triplet \((X, p, G)\) in Lemma~\ref{nil}, given \(r > 0\) and a subset \(X' \subseteq X\), define
\[
G(X', r) \coloneqq \{ g \in G \mid d(gx, x) \le r \ \text{for all } x \in X' \}.
\]

\begin{lemma}\label{lem:local_diffeo}
There exist constants \(r, \epsilon > 0\), depending only on the limit \((X, p, G)\), such that:
\begin{enumerate}[(i)]
    \item \(G(B_r(p), 10 \epsilon) \subseteq G_0\);
    \item \(G(B_r(p), 10 \epsilon)\) contains no non-trivial subgroup;
    \item There exists a closed neighborhood \(U\) of the origin in the Lie algebra \(\mathfrak{g}\) of \(G_0\) such that the exponential map
    \[
    \exp \colon U \longrightarrow G(B_r(p), 10 \epsilon)
    \]
    is a diffeomorphism onto its image.
\end{enumerate}
\end{lemma}

\begin{proof}
Since \(G\) is a Lie group, there exists a neighborhood \(V \subseteq G_0\) of the identity satisfying properties (i)--(iii) (for example, small enough to avoid non-trivial subgroups and to ensure the exponential map is a diffeomorphism).  
It suffices, then, to choose \(r > 0\) and \(\epsilon > 0\) small enough so that
\[
G(B_r(p), 10 \epsilon) \subseteq V.
\]

Assume for contradiction that no such \(r\) and \(\epsilon\) exist. Then there is a sequence \(g_i \in G\) with \(g_i \in G(B_i(p), 1/i)\) but \(g_i \notin V\).  
The condition \(g_i \in G(B_i(p), 1/i)\) implies \(g_i \to e\) uniformly on compact subsets of \(Y\). In the compact-open topology of \(\mathrm{Iso}(X)\), this convergence is equivalent to \(g_i \to e\) in \(G\), which contradicts \(g_i \notin V\) for all \(i\).  

Hence, there exist suitable constants \(r\) and \(\epsilon\) satisfying the lemma.
\end{proof}

Choose a sequence \(\epsilon_i \to 0\) such that the eGH distance between \((X_i, p_i, G_i)\) and \((X, p, G)\) is at most \(\epsilon_i\).

\begin{lemma}[Gap Lemma]\label{lem:gap_lemma}
Let \(r, \epsilon > 0\) be as in Lemma~\ref{lem:local_diffeo}. Then there exists a sequence \(\epsilon_i' \to 0\) such that, for all sufficiently large \(i\),
\[
\langle G_i(B_r(p_i), \epsilon_i') \rangle = \langle G_i(B_r(p_i), \epsilon) \rangle.
\]
\end{lemma}

\begin{proof}
To establish the equality, it suffices to show that for any fixed \(\delta \in (0, \epsilon)\),
\[
G_i(B_r(p_i), \epsilon) \subseteq \langle G_i(B_r(p_i), \delta) \rangle
\]
for all sufficiently large \(i\).

Suppose, for contradiction, that there exists a sequence of elements
\[
g_i \in G_i(B_r(p_i), \epsilon) \quad \text{with} \quad g_i \notin \langle G_i(B_r(p_i), \delta) \rangle.
\]
Passing to a subsequence if necessary, we may assume \(g_i \to g \in G(B_r(p), \epsilon)\). By Lemma~\ref{lem:local_diffeo}, there exists a unique vector \(v \in U \subset \mathfrak{g}\) such that \(\exp(v) = g\).

Choose an integer \(N\) large enough so that
\[
h \coloneqq \exp(v/N) \in G(B_r(p), \delta/2),
\]
ensuring that \(h^N = g\). Let \(h_i \in G_i\) be a sequence converging to \(h\). For sufficiently large \(i\), we then have \(h_i \in G_i(B_r(p_i), \delta)\), and \(h_i^N\) is \(N\epsilon_i\)-close to \(g_i\).

It follows that
\[
g_i^{-1} h_i^N \to e \in G.
\]
Hence, for large \(i\), \(g_i^{-1} h_i^N\) lies in \(G_i(B_r(p_i), \delta)\).  
Then we conclude that \(g_i \in \langle G_i(B_r(p_i), \delta) \rangle\), contradicting our assumption.  
\end{proof}

To simplify notations, we may assume \(\epsilon_i' = \epsilon_i \to 0\) in Lemma~\ref{lem:gap_lemma}. Passing to a subsequence if necessary, we may assume that the generated subgroups converge to a limit \(G_{\epsilon} \subseteq G\):
\[
\langle G_i(B_r(p_i), \epsilon) \rangle \xrightarrow{eGH} G_{\epsilon}.
\]
By construction, \(G_{\epsilon}\) contains an \(\epsilon\)-neighborhood of the identity and therefore contains $G_0$.

\begin{proof}[Proof of Lemma \ref{nil}]
To analyze the structure of \(G_0\), we may assume without loss of generality that \(G = G_{\epsilon}\) and that each \(G_i\) is generated by
\[
S_i \coloneqq G_i(B_r(p_i), \epsilon_i).
\]
For convenience, we also assume that \(\epsilon > 100 \, \epsilon_i\).

Let \(k \in \mathbb{N}\) be such that \(G(B_r(p), 4\epsilon)\) can be covered by \(k\) left translates of \(G(B_r(p), \epsilon/2)\). We claim that
\[
A_i \coloneqq G_i(B_r(p_i), \epsilon)
\]
is a \(k\)-approximate group for sufficiently large \(i\), i.e., \(A_i^2\) can be covered by \(k\) translates of \(A_i\).

By the choice of \(k\), there exist elements \(g_1, \dots, g_k \in G(B_r(p), 10\epsilon)\) such that
\[
G(B_r(p), 4\epsilon) \subseteq \bigcup_{j=1}^k g_j G(B_r(p), \epsilon/2).
\]
Let \(g_1', \dots, g_k' \in G_i(B_r(p_i), 11\epsilon)\) be sequences approximating these elements. Then, for sufficiently large \(i\),
\[
A_i^2 \subseteq \bigcup_{j=1}^k g_j' A_i.
\]

Indeed, for any \(h_i \in A_i^2 \subseteq G_i(B_r(p_i), 3\epsilon)\), there exists a limit element \(h \in G(B_r(p), 4\epsilon)\) with \(d(h_i, h) < \epsilon_i\). By construction, \(h = g_j l\) for some \(j\) and some \(l \in G(B_r(p), \epsilon/2)\). Choosing \(l_i \in G_i(B_r(p_i), 2\epsilon/3)\) close to \(l\), we find that \(h_i\) is \(10\epsilon_i\)-close to \(g_j' l_i\). Consequently, there exists \(t \in G_i(B_r(p_i), 10\epsilon_i)\) such that
\[
h_i = g_j' l_i t.
\]
Since \(l_i t \in A_i\), the claimed inclusion follows.

Applying Theorem~\ref{thm:BGT_structure} to the pair \((A_i, S_i)\) produces a normal subgroup \(G_i' \trianglelefteq G_i\) of index at most \(C(k)\). Its limit is a normal subgroup of \(G\) with finite index, hence it must contain \(G_0\). Passing to this subgroup, we may assume \(G_i' = G_i\).

Theorem~\ref{thm:BGT_structure} also provides a normal subgroup \(N_i \trianglelefteq G_i\) with \(N_i \subseteq A_i^4\) such that \(G_i / N_i\) is nilpotent of step at most \(C(k)\). As \(i \to \infty\), \(N_i\) converges to a subgroup \(N \subseteq G(B_r(p), 10\epsilon)\). By Lemma~\ref{lem:local_diffeo}, \(N\) is trivial.

Thus, we obtain the following equivariant Gromov--Hausdorff convergence:
\[
\begin{CD}
(X_i, p_i, G_i, N_i) @>eGH>> (X, p, G, \{e\}) \\
@VV\pi V @VV\pi V \\
(X_i / N_i, \bar{p}_i, G_i / N_i) @>eGH>> (X, p, G).
\end{CD}
\]

Since each \(G_i / N_i\) is nilpotent of step at most \(C(k)\), their limit \(G\) is also nilpotent. In particular, the identity component \(G_0 \subseteq G\) is nilpotent, completing the proof.
\end{proof}

Since \(G_0\) is a nilpotent Lie group, its maximal compact subgroup \(T^k\) is a torus contained in the center of \(G_0\).

\begin{corollary}\label{cor:torus_quotient}
Under the hypotheses of the Main Theorem, the quotient space \(X/T^k\) is simply connected.
\end{corollary}

\begin{proof}
Theorem~\ref{thm:G_0} establishes that \(X/G_0\) is simply connected. To relate this to the torus quotient, consider the quotient group \(G_0/T^k\). For a connected nilpotent Lie group, this quotient is a simply connected Lie group diffeomorphic to some Euclidean space \(\mathbb{R}^l\).

The group \(G_0/T^k\) acts on \(X/T^k\), and the resulting quotient satisfies
\[
(X/T^k)/(G_0/T^k) \cong X/G_0.
\]
Since \(T^k\) is the maximal compact subgroup of \(G_0\) and isotropy subgroups of the \(G_0\)-action are compact, the induced action of the torsion-free group \(G_0/T^k\) on \(X/T^k\) is free. Consequently, we obtain a fiber bundle
\[
G_0/T^k \longrightarrow X/T^k \longrightarrow X/G_0.
\]

Finally, since \(G_0/T^k \cong \mathbb{R}^l\) is contractible and \(X/G_0\) is simply connected, it follows that \(X/T^k\) is also simply connected.
\end{proof}

\section{Connected Lie group actions}

In the setting of Main~Theorem, we have shown that $X/T^k$ is simply connected. To study the fudamental group of $X$, it is enough to prove Lemma~\ref{lem:lifting_compact}. 


\begin{lemma}[Lemma~\ref{lem:lifting_compact}]\label{compact}
Let $(Y,p)$ be a semi-locally simply connected proper geodesic metric space, and suppose that a connected Lie group $K$ acts isometrically on $Y$, with quotient $(Y/K, \bar{p})$. Then there is a natural surjective homomorphism
\[
\pi_* \colon \pi_1(Y,p) \longrightarrow \pi_1(Y/K, \bar{p}),
\]
whose kernel is generated, up to conjugation, by loops contained entirely in the $K$-orbits.
\end{lemma}

We say that a loop \(\beta\) is \emph{contained in a \(K\)-orbit} if its image \(\mathrm{Im}(\beta)\) lies entirely in the \(K\)-orbit of its basepoint \(\beta(0)\). Moreover, a loop based at \(p\) of the form
\[
\alpha \cdot \beta \cdot \alpha^{-1}
\]
is said to be \emph{contained in a \(K\)-orbit up to conjugation} if \(\beta\) is contained in some \(K\)-orbit and \(\alpha\) is a path connecting the basepoint \(p\) to \(\beta(0)\).

The proof of Lemma~\ref{compact} relies on the following technical result, which we establish later in this section.

\begin{lemma}\label{lem:orbit_move}
Under the same assumptions as Lemma~\ref{compact}, for any \(R > 0\) there exists a constant \(\delta_0 > 0\) with the following property:

Let \(c_1, c_2 \subset B_R(p)\) be loops based at \(p\). If their projections under the quotient map \(\pi \colon Y \to Y/K\) satisfy
\[
d\big(\pi(c_1(t)), \pi(c_2(t))\big) \le \delta_0 \quad \text{for all } t \in [0,1],
\]
then the loop \(c_1 \cdot c_2^{-1}\) is homotopic to a product of loops, each of which is contained in a \(K\)-orbit up to conjugation.
\end{lemma}

\begin{proof}[Proof of Lemma \ref{compact}]
Let \(\gamma \subset Y\) be a loop based at \(p\), and let \(\bar{\gamma} = \pi \circ \gamma\) denote its projection into \(Y/K\). The map
\[
\pi_*([\gamma]) = [\bar{\gamma}] \in \pi_1(Y/K, \bar{p})
\]
is a well-defined group homomorphism induced by the continuous quotient \(\pi \colon Y \to Y/K\).

We first prove that $\pi_*$ is surjective.  
Given any loop \(\bar{\gamma}\) based at \(\bar{p}\) in \(Y/K\), the path-lifting property allows us to lift \(\bar{\gamma}\) to a path \(\tilde{\gamma}\) in \(Y\) starting at \(p\) and ending at some point \(q\). Since \(\pi(p) = \pi(q) = \bar{p}\), there exists an element \(g \in K\) such that \(gp = q\).  

Because \(K\) is path-connected, we may choose a continuous path \(g(t) \in K\) with \(g(0) = e\) and \(g(1) = g^{-1}\). Defining the modified lift
\[
\sigma(t) \coloneqq g(t) \tilde{\gamma}(t),
\]
we obtain a loop in \(Y\) based at \(p\). Moreover,
\[
\pi(\sigma(t)) = \pi(g(t)\tilde{\gamma}(t)) = \pi(\tilde{\gamma}(t)) = \bar{\gamma}(t),
\]
so \(\pi_*([\sigma]) = [\bar{\gamma}]\), which confirms that \(\pi_*\) is surjective.

Now, we study the kernel of $\pi_*$. Let $\gamma$ be a loop at $p$ whose projection $\bar{\gamma}$ is contractible in $Y/K$. Let
\[
H \colon [0,1] \times [0,1] \to Y/K
\]
be a null-homotopy of $\bar{\gamma}$, so that 
\[
H(t,0) = \bar{\gamma}(t) \quad \text{and} \quad H(t,1) = \bar{p}.
\] 
Choose $R>0$ large enough such that $\mathrm{Im}(H) \subset \pi(B_R(p))$. Then there exists $\delta_0>0$ for which Lemma~\ref{lem:orbit_move} applies.  

By the uniform continuity of $H$, we may select a partition
\[
0 = s_0 < s_1 < \cdots < s_L = 1
\]
such that 
\[
d(H(t,s_l), H(t,s_{l+1})) \le \delta_0 \quad \text{for all } t \in [0,1] \text{ and } 0 \le l \le L-1.
\]

Set $\gamma_0 = \gamma$. For each $l = 1, \dots, L$, let $\gamma_l$ be a loop at $p$ lifting the path $H(\cdot, s_l)$. At the final step, $\gamma_L$ projects to the constant loop $\bar{p}$, so $\gamma_L$ lies entirely in the $K$-orbit of $p$.  

Applying Lemma~\ref{lem:orbit_move} to each consecutive pair $(\gamma_l, \gamma_{l+1})$, we see that the loop
$\gamma_l \cdot \gamma_{l+1}^{-1}$
is homotopic to a product of loops contained in $K$-orbits, up to conjugation. Writing $\gamma$ as the telescoping product
\[
\gamma = (\gamma_0 \cdot \gamma_1^{-1}) \cdot (\gamma_1 \cdot \gamma_2^{-1}) \cdots (\gamma_{L-1} \cdot \gamma_L^{-1}) \cdot \gamma_L,
\]
we conclude that $\gamma$ is homotopic to a product of loops, each contained in a $K$-orbit up to conjugation.
\end{proof}

We now present the proof of the technical result, Lemma~\ref{lem:orbit_move}.

Let $c \colon [0,1] \to Y$ be a curve and fix $t_0 \in (0,1)$. Denote by $c|_{[0,t_0]}$ the restriction of $c$ to the interval $[0,t_0]$. For any loop $c'$ lying entirely in the $K$-orbit of the point $c(t_0)$, define a new curve $c''$ by
\[
c'' \coloneqq (c|_{[0,t_0]}) \cdot c' \cdot (c|_{[t_0,1]}) 
       = (c|_{[0,t_0]}) \cdot c' \cdot (c|_{[0,t_0]})^{-1} \cdot c.
\]
After reparametrizing $c''$ to the interval $[0,1]$, we say that $c''$ is obtained by \emph{gluing} the orbit-loop $c'$ to $c$ at the point $c(t_0)$.

With this construction, the product of the new curve with the inverse of the original,
\[
c'' \cdot c^{-1} \simeq (c|_{[0,t_0]}) \cdot c' \cdot (c|_{[0,t_0]})^{-1},
\]
is a loop based at $c(0)$ that lies entirely in a $K$-orbit, up to conjugation.

\begin{proof}[Proof of Lemma \ref{lem:orbit_move}]
Let $U_{\delta} \subset K$ denote the set of elements that are $\delta$-close to the identity $e \in K$. Since $K$ is a Lie group, for sufficiently small $\delta > 0$, the set $U_{\delta}$ contains no nontrivial subgroup of $K$. We may assume that $U_{\delta}$ is connected; otherwise we choose $\delta' < \delta$ such that $U_{\delta'}$ lies in the identity component of $U_{\delta}$.

$Y$ is semi-locally simply connected. Then we can take $\epsilon > 0$ such that for any $y \in B_R(p)$, every loop contained in $B_{3\epsilon}(y)$ is contractible in $Y$.

Since the $K$-action is isometric, and hence uniformly continuous on $B_R(p)$, we may choose $\delta > 0$ sufficiently small so that for any $g_1, g_2 \in K$ and $y_1, y_2 \in B_R(p)$, if
\[
g_1 g_2^{-1} \in U_{\delta} \quad \text{and} \quad d(y_1, y_2) < \delta,
\]
then
\[
d(g_1 y_1, g_2 y_2) < \epsilon.
\]

We claim that there exists $\delta_0 > 0$ such that for any $g \in K$ and $y \in B_R(p)$, if $d(gy, y) < 10 \delta_0$, then there exist $y_0$ $\delta$-close to $y$ and $g_0 \in K_{y_0}$ such that
\[
g_0 g^{-1} \in U_{\delta}.
\]

Suppose, for contradiction, that no such $\delta_0$ exists. Then there exist sequences $y_i \in B_R(p)$ and $g_i \in K$ such that
\[
d(g_i y_i, y_i) < \frac{1}{i},
\]
but no pair $(y_0, g_0)$ satisfies the conclusion. Passing to a subsequence if necessary, we may assume that $y_i \to y$ and $g_i \to g$. It follows that $g y = y$. For sufficiently large $i$, we then have $g^{-1} g_i \in U_{\delta}$, which yields a contradiction.

We now show that the constant $\delta_0$ defined above satisfies the lemma. Assume that
\[
d(\pi(c_1(t)), \pi(c_2(t))) \le \delta_0 \quad \text{for all } t \in [0,1].
\]
Choose a subdivision
\[
0 = t_0 < t_1 < \cdots < t_L = 1
\]
such that for each $1 \le l \le L$, diameters of $c_1([t_{l-1}, t_l])$ and $c_2([t_{l-1}, t_l])$ are less than $\delta_0$.

For each $0 \le l \le L$, since
\[
d(\pi(c_1(t_l)), \pi(c_2(t_l))) \le \delta_0,
\]
there exists $g_l \in K$ such that
\begin{equation}\label{e6}
d(c_1(t_l), g_l c_2(t_l)) \le \delta_0.
\end{equation}
We may choose $g_0 = g_L = e \in K$, since $p = c_1(0) = c_2(0) = c_1(1) = c_2(1)$.

The idea of the proof is to construct a new loop $c_2'$ by modifying $c_2$ locally within finitely many $\epsilon$-balls and, if necessary, attach loops in $K$-orbits to $c_2$. Then we shall construct a continuous map
\[
g(t) \colon [0,1] \to K, \quad g(0) = g(1) = e,
\]
such that,
\[
d(c_1(t), g(t)\, c_2'(t)) \le \epsilon \quad \text{for all } t \in [0,1].
\]

Assuming such a map $g(t)$ exists, we define a homotopy
\[
H(s,t) := g(st)\, c_2'(t).
\]
Then $H(0,t) = c_2'(t)$, $H(1,t)$ is pointwise $3\epsilon$-close to $c_1(t)$, $H(s,0)=p$ and $H(s,1)=g(s)p$. It follows that $c_1 \cdot c_2^{-1}$ is homotopic to the loop $g(s)p$ which is contained in the $K$-orbit of $p$. Therefore, it suffices to construct such $c_2'(t)$ and $g(t)$.

We begin by constructing $c_2'(t)$ and $g(t)$ on $[0,t_1]$. By the triangle inequality and \eqref{e6}, we have
\begin{align*}
d(g_1 c_2(t_1), c_2(t_1))
&\le d(g_1 c_2(t_1), c_1(t_1)) + d(c_1(t_1), c_1(t_0)) \\
&\quad + d(c_1(t_0), c_2(t_0)) + d(c_2(t_0), c_2(t_1)) \\
&< 10 \delta_0.
\end{align*}

By the choice of $\delta_0$, there exist $y_1 \in B_\delta(c_2(t_1))$ and $g_1' \in K_{y_1}$ such that
\[
g_1^{-1} g_1' \in U_{\delta}.
\]

We now define a path $c_2''$ on $[0,t_1]$ by joining $c_2(0)$ to $y_1$ by a geodesic, and then $y_1$ to $c_2(t_1)$ by another geodesic. This new path is homotopic to $c_2$, since both are contained in a ball of radius $\epsilon$.

Since $K$ is a connected Lie group, there exists a continuous path $h \colon [t_1/3, 2t_1/3] \to K$ such that $h(t_1/3)=e$ and $h(2t_1/3)=g_1'$. We then attach a loop in the $K$-orbit of $y_1$ to $c_2$, defined by
\[
\gamma(t) := h(t)^{-1}y_1, \quad t \in [t_1/3, 2t_1/3].
\]
$\gamma(t)$ is a loop since $g_1' \in K_{y_1}$.

We define $c_2'$ on $[0, t_1]$ as follows:
\begin{itemize}
    \item For $t \in [0, t_1/3]$, let $c_2'(t)$ follow the geodesic from $c_2(0)$ to $y_1$.
    \item For $t \in [t_1/3, 2t_1/3]$, let $c_2'(t)=\gamma(t)$.
    \item For $t \in [2t_1/3, t_1]$, let $c_2'(t)$ follow the geodesic from $y_1$ to $c_2(t_1)$.
\end{itemize}
Then $c_2'(0)=c_2(0)$ and $c_2'(t_1)=c_2(t_1)$.
The resulting curve $c_2'$ on $[0, t_1]$ is homotopic to the original $c_2$ with the addition of $\gamma(t)$ contained in the $K$-orbit of $y_1$.

Next, we define $g \colon [0, t_1] \to K$ as follows:
\begin{itemize}
    \item $g(t) = e$ for $t \in [0, t_1/3]$.
    \item For $t \in [t_1/3, 2t_1/3]$, define $g(t)=h(t)$. In particular, $g(t)\, c_2'(t) = y_1$ and $g(2t_1/3)= g_1'$.
    \item For $t \in [2t_1/3, t_1]$, define $g(t)$ continuously so that $g(t) g_1^{-1} \in U_{\delta}$ and $g(t_1) = g_1$. This is possible since $g_1^{-1} g_1' \in U_{\delta}$ and $U_{\delta}$ is connected.
\end{itemize}

We now verify that $g(t) c_2'(t)$ is pointwise $3\epsilon$-close to $c_1$ on $[0, t_1]$.

For $t \in [0, t_1/3]$, we have
\begin{align*}
d(g(t) c_2'(t), c_1(t))
&= d(c_2'(t), c_1(t)) \\
&\le d(c_2'(t), c_2'(0)) + d(c_2'(0), c_1'(0)) + d(c_1(0), c_1(t)) \\
&< 2\delta < \epsilon.
\end{align*}

For $t \in [t_1/3, 2t_1/3]$, we have $g(t) c_2'(t) = y_1$, and
\[
d(y_1, c_1(t)) < 2\delta < \epsilon.
\]

For $t \in [2t_1/3, t_1]$, since $g(t) g_1^{-1} \in U_{\delta}$, we obtain
\begin{align*}
d(g(t) c_2'(t), c_1(t))
&\le d(g(t) c_2'(t), g_1 c_2'(t)) + d(g_1 c_2'(t), g_1 c_2'(t_1)) \\
&\quad + d(g_1 c_2'(t_1), c_1(t_1)) + d(c_1(t_1), c_1(t)) \\
&\le \epsilon + \delta + \delta_0 + \epsilon < 3\epsilon.
\end{align*}


We repeat the above construction on each subinterval $[t_l, t_{l+1}]$. Notice that $g_0=e$ is not special since we may construct for $g_l^{-1} c_2(t)$ and $g_l^{-1}g_{l+1}$ on $[t_l,t_{l+1}]$ as well; then we multiply $g_l$. At each step, we construct $c_2'$ by modifying within small balls and, if necessary, attach a loop in a $K$-orbit. Moreover $c_2'(t_l)=c_2(t_l)$ for each $l$. 

We also obtain a continuous map $g \colon [0,1] \to K$ such that $g(t_l)=g_l$ and
\[
d(g(t) c_2'(t), c_1(t)) \le 3\epsilon \quad \text{for all } t \in [0,1], \qquad \text{and} \qquad g(t_l) = g_l \text{ for each } l.
\]

\end{proof}
\bibliographystyle{plain} 
\bibliography{bib}
\end{document}